\documentclass[a4paper,twoside]{article}
\usepackage{amsthm,amscd}

\usepackage{amsmath,amssymb,amsfonts,stmaryrd}
\usepackage[T1]{fontenc}
\usepackage{tikz}
\usetikzlibrary{trees}
\usepackage{pdfpages}

\usepackage[numbers]{natbib}

\usepackage[utf8]{inputenc}
\usepackage[english]{babel}
\usepackage{graphicx}               
\usepackage{color}                  
\RequirePackage[colorlinks=true,citecolor=blue,linkcolor=blue]{hyperref}
\usepackage{hyperref}
\usepackage[absolute]{textpos} 
\usepackage{multicol}
\usepackage{array}
\usepackage{bbm}

\usepackage[babel=true]{csquotes} 

\newtheorem{theorem}{Theorem}[section]
\newtheorem{cor}[theorem]{Corollary}
\newtheorem{lemma}[theorem]{Lemma}
\newtheorem{prop}[theorem]{Proposition}
\newtheorem{proptes}[theorem]{Properties}

\theoremstyle{definition}
\newtheorem{definition}[theorem]{Definition}

\newtheorem{rmk}[theorem]{Remark}

\def\leq{\leqslant}
\def\geq{\geqslant}

\newcommand{\abs}[1]{\left\lvert #1 \right\rvert}

\newcommand{\defeq}{\mathrel{\mathop:}=}

\newcommand{\ind}{\mathbbm{1}}

\newcommand{\beq} {\begin{eqnarray*}}
\newcommand{\eeq} {\end{eqnarray*}}
\newcommand{\trm} {\textrm}
\newcommand{\tbf} {\textbf}
\newcommand{\noi} {\noindent}

\def \R{\mathbb{R}}

\def \N{\mathbb{N}}
\def \Z{\mathbb{Z}}
\def \E{\mathbb{E}}

\def \Var{\hbox{{\rm Var}}}

\def \Cov{\hbox{{\rm Cov}}}

\usepackage{hyperref}

\title{Semi-parametric estimation of the variogram scale parameter   of a Gaussian process  with stationary increments} 

\author{Jean-Marc Aza\"is\footnote{Institut de Math\'ematiques de Toulouse; UMR5219. Université de Toulouse; CNRS. UT3, F-31062 Toulouse, France.}
 \and Fran\c cois Bachoc$^*$
 \and Agn\`es~Lagnoux\footnote{Institut de Math\'ematiques de Toulouse; UMR5219. Université de Toulouse; CNRS. UT2J, F-31058 Toulouse, France. }
 \and Thi Mong Ngoc Nguyen\footnote{VNUHCM - University of Science, Ho Chi Minh city, Viet Nam.}
}
\begin{document}

\maketitle
\begin{abstract}
We consider the semi-parametric estimation of the scale parameter of the variogram of a one-dimensional Gaussian
process with known smoothness. We suggest an estimator based on quadratic variations and on the
moment method. We provide asymptotic approximations of the mean and variance of this estimator,
together with asymptotic normality results, for a large class of Gaussian processes. We allow for
general mean functions and study the aggregation of several estimators based on various variation
sequences. In extensive simulation studies, we show that the asymptotic results accurately depict the
finite-sample situations already for small to moderate sample sizes. We also compare various variation
sequences and highlight the efficiency of the aggregation procedure.
\end{abstract}

\section{Introduction}\label{sec:intro}

\paragraph{General context and state of the art} Gaussian process models are widely used in statistics. For instance, they enable to interpolate observations by Kriging, notaby in computer experiment designs to build a metamodel \cite{rasmussen06gaussian,stein99interpolation}. A second type of application of Gaussian processes is the analysis of local characteristics of images \cite{richard2018anisotropy} and one dimensional signals (e.g.\ in finance, see \cite{wu2014,han2015financial} and the references therein). A central problem with Gaussian processes is the estimation of the covariance function or the variogram.
In this paper, we consider a real-valued Gaussian process $(X(t))_{t\in \R}$ with stationary increments. Its semi-variogram is well-defined and given by 
\begin{align}\label{def:variogram}
V(h):=\frac12\E\left[\left(X(t+h)-X(t)\right)^2\right].
\end{align}
Ideally, one aims at knowing perfectly the function $V$ or at least estimate it precisely, either in a parametric setting or in a nonparametric setting. The parametric approach consists in assuming that the mean function of the Gaussian process (the drift) is a linear combination of known functions (often polynomials) and that the semi-variogram $V$ belongs to a parametric family of semi-variograms $\{V_\theta,\, \theta \in \Theta \subset \R^p\}$ for a given $p$ in $\N^*$. Furthermore, in most practical cases, the semi-variogram is assumed to stem from a stationary autocovariance function $k$ defined by $k(h)=\Cov(X(t),X(t+h))$. In that case, the process is supposed to be stationary, and $V$ can be rewritten in terms of the process autocovariance function $k$: $V(h)=k(0)-k(h)$.
Moreover, a parametric set of stationary covariance functions is considered of the form 
$\{  k_{\theta} , \theta \in \Theta \}$ with $\Theta \subset \mathbb{R}^p$. 
In such a setting, several estimation procedures for $\theta$ have been introduced and studied in the literature. Usually in practice, most of the software  packages (like, e.g.\ \texttt{DiceKriging} \cite{RGD12}) use  the maximum likelihood estimation method (MLE) to estimate $\theta$ (see \cite{stein99interpolation,SWN03,rasmussen06gaussian} for more details on MLE). Unfortunately, MLE is known to be computationally expensive and intractable for large data sets. In addition, it may diverge  in some  complicated situations (see Section \ref{ssec:real:data}). This has motivated the search for alternative estimation methods with a good balance between computational complexity and statistical efficiency. Among these methods, we can mention low rank approximation \cite{stein14limitations}, sparse approximation \cite{hensman2013}, covariance tapering \cite{furrer2006covariance,kaufman08covariance}, Gaussian Markov random fields approximation \cite{datta16hierarchical,rue05gaussian}, submodel aggregation \cite{caoGPOE,deisenroth2015,hinton2002training,rulliere2018nested,
trespBCM,van2015optimally} and composite likelihood \cite{BL19}.

\paragraph{Framework and motivation} The approaches discussed above are parametric.  In this paper, we consider a more general semi-parametric context. The Gaussian process $X$ is only assumed to have stationary increments and no parametric assumption is made on the semi-variogram $V$ in \eqref{def:variogram}.
Assume, to simplify, that  the semi-variogram is a $C^{\infty}$ function  outside  0. This is the case for most of the models  even if the sample paths are not regular, see the examples in Section \ref{ssec:ass:proc}.  Let $D$ be the order of differentiability in quadratic mean  of $(X(t))_{t\in \R}$. This is equivalent to the fact  that  $V$ is $2D$ differentiable  and not $2D+2$ differentiable.  Let us assume that the $2 D$'th derivative of $V$ has the following expansion at the origin:
\begin{align}\label{eq:expansion}
V^{(2D)}(h) =  V^{(2D)}(0)  +C  (-1)^{D}  \abs{h}^{s} +r(h),
\end{align}
where $C\geqslant 0$, $0<s<2$, and the remainder function $r$ satisfies some hypothesis detailed further (see Section \ref{ssec:ass:proc}) and is a $o(|h|^s)$ as $h \to 0$. Note that, since $s<2$,  $V$ is indeed not $(2D+2)$ differentiable. The quantity $s$ is the smoothness parameter and we call $C$ the scale parameter. In this paper, we assume $D$ and $s$ to be known and we focus on the theoretical study of the semi-parametric estimation of $C$ defined in \eqref{eq:expansion} in dimension one. Remark that this makes it possible to test whether the
Gaussian process stems from a white noise or not. Notice that we also perform some additional simulations in higher dimensions. The value of $C$ may lead to significantly different behaviors of the process $X$ as one can see in Figure \ref{fig:exemple_motiv} which represents several realizations of a Gaussian process with exponential covariance function for different values of $C$ ($V(h) = 1 - \exp(-C|h|)$, which satisfies \eqref{eq:expansion} with $D=0$). More concretely, for instance when $D = 0$, $C$ provides the first order approximation of $\E\left[\left(X(t+h)-X(t)\right)^2\right]$ when $h$ is small.
Moreover, when $D=0$, $C=+\infty$ traduces independence, i.e.\ the process $X$ reduces to a white noise whereas $C=0$ corresponds to a constant process $X$. 

\begin{figure}
\centering
\includegraphics[width=16cm]{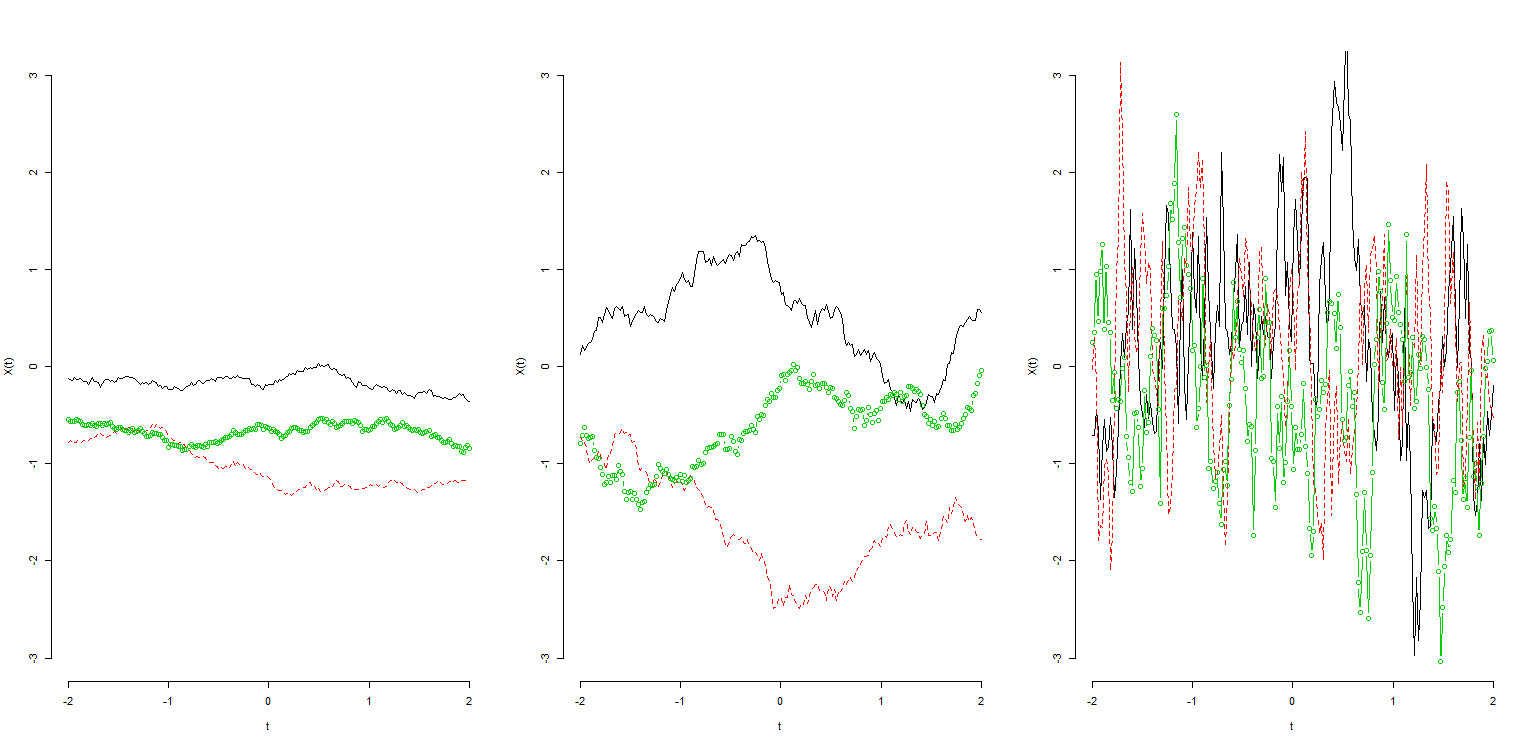}
\caption{Several realizations of Gaussian processes with exponential covariance function with parameter $C=0.1$ (left), 10 (middle), and 100 (right).
}\label{fig:exemple_motiv}
\end{figure}

As a motivating example, consider the following case where the estimation of $C$ is beneficial.
Assume that we observe a 
signal $S$ depending on a vector of initial parameters $x$ given by a computer code described by the application:
\begin{eqnarray}\label{def:code}
S  : & E  & \to  \R^{[0,1]}\\
& x & \mapsto  (S_x(t),\,t\in[0,1]),\nonumber
\end{eqnarray}
where $E$ stands for the initial parameter space. 
In order to interpret the output curve $t \mapsto S_x(t)$ for a given parameter $x$, it is useful to consider that this curve is the realization of a Gaussian process, with scale parameter $C_x$. It is then insightful to study the black-box $x \mapsto C_x$, for instance by means of a sensitivity analysis, or in the aim of finding which inputs $x$ lead to $C_x <+\infty$, that is to output curves with a dependence structure. A necessary step to such a study is the estimation of the value of $C_x$, given a discretized version of the curve $t \mapsto S_x(t)$.

More generally, estimating $C$ enables to assess if an observed signal is composed of independent components ($C = + \infty$) or not ($C < \infty$), and to quantify the level of dependence. We refer to the real data sets studied in Section \ref{ssec:real:data} for further discussion.

\paragraph{State of the art on variogram estimation} Nonparametric estimation of the semi-variogram function is a difficult task since the resulting estimator must necessarily lead to a ``valid'' variogram (conditional negative definiteness property) \cite[p. 93]{cressie93}. This requirement usually leads to complicated and computationally involved nonparametric
estimators of the variogram \cite{hall1994nonparametric,hall1994properties} that may need huge data sets to be meaningful. A simpler estimator based on the moment method has been proposed in \cite{Matheron62, cressie1980robust} but does not always conduce to a valid variogram. 
A classical approach to tackle this problem has been proposed in the geostatistics literature \cite{david2012geostatistical, journelj, clark1979practical} and consists in fitting a parametric
model of valid semi-variograms to a pointwise nonparametric semi-variogram estimator by minimizing
a given distance between the nonparametric estimator and the semi-variograms at a finite number of lags. The reader is refered to \cite[Chapter 2]{cressie93} for further details on semi-variogram model fitting and to \cite{LLC02} for least-squares methods.

\paragraph{State of the art on quadratic variations} In order to remedy the drawbacks of the MLE previously mentioned, we focus on an alternative estimation method using quadratic variations based on the observations of the process $X$ at a triangular array of points $(t_j)_j$, where $j=1,...,n$. This method is also an alternative to the existing methods mentioned in the previous paragraph. Quadratic variations have been first introduced by Levy in \cite{Levy40} to quantify the oscillations of the Brownian motion. 
Then a first result on the quadratic variation of a Gaussian non-differentiable process is due to Baxter (see e.g. \cite{Baxter56}, \cite[Chap. 5]{Grenander81} and \cite{Gladyshev61}) that ensures (under some conditions) the almost sure convergence (as $n$ tends to infinity) of 
\begin{equation}\label{def:V_1n}
 \sum_{j=1}^{n} \left(X(t_j)-X(t_{j-1}\right)^2,
\end{equation}
where $t_j=j/n$, for $j=1,...,n$ (by convention $t_0=0$ and $X_0=0$). 
A  generalization of the previous quadratic variations has been introduced in Guyon and Léon \cite{GL89}: for a given real function $H$, the $H$-variation is given by
\begin{equation}\label{def:V_Hn}
V_{H,n}\defeq \sum_{j=1}^{n} H\left(\frac{X(t_j)-X(t_{j-1})}{(\Var(X(t_j)-X(t_{j-1}))^{1/2}}\right).
\end{equation}
In \cite{GL89}, technical conditions are assumed and a smoothness parameter $0 < s <2$, similar to the one in \eqref{eq:expansion} when $D=0$, is considered.
Then the most unexpected result of \cite{GL89} is that $(V_{H,n}/n)_n$ has a limiting normal distribution with convergence rate $n^{1/2}$  when $0< s  <3/2$ whereas the limiting distribution is non normal and the convergence rate is reduced to $n^{2-s}$ when $3/2< s  <2$.
Moreover, for statistical purposes, it has been proved by Coeurjolly that quadratic variations  are optimal (details and precisions can be found in \cite{Coeurjolly01}).   
In \cite{IL97}, Istas and Lang generalized the results on quadratic variations. They allowed for observation points of the form $t_j=j\delta_n$ for $j=1,\dots, n$, with $\delta_n$ depending on $n$ and tending to 0 as $n$ goes to infinity. 
They studied the generalized quadratic variations defined by:
\begin{equation}\label{def:V_A}
V_{a,n}\defeq \sum_{i=1}^{n-1} \left( \sum_k a_k X(i+k \delta_n) \right)^2,
\end{equation}
where the sequence $a = (a_k)_k$ has a finite support and some vanishing moments. 
Then they built estimators of the smoothness parameter and the scale parameter $C$ and showed that these estimators are almost surely consistent and asymptotically normal. In the more recent work of Lang and Roueff \cite{LR01}, the authors 
generalized the results of Istas and Lang \cite{IL97} and Kent and Wood \cite{KW97} on an increment-based estimator in a semi-parametric framework with different sets of hypothesis.
Another generalization for non-stationary Gaussian processes and quadratic variations along curves is done in \cite{AP93}. See also the studies of \cite{Perrin99} and \cite{Coeurjolly01}.

\paragraph{Contributions of the paper} Now let us present the framework considered in our paper. 
We assume that the Gaussian process $X$ has stationary increments and is observed at times $t_j=j\delta_n$ for $j=1,\dots,n$ with $\delta_n$ tending to zero.
Note that $t_j$ also depends on $n$ but we omit this dependence in the notation for simplicity.
We will only consider $\delta_n = n^{-\alpha}$  with $0 < \alpha \leq 1$  throughout the article.
Two cases are then considered: $\alpha = 1$ ($\delta_n = 1/n$, infill asymptotic setting \cite{cressie93}, that we call the
infill situation throughout) and $0<\alpha<1$ ($\delta_n \to 0$ and $n \delta_n \to \infty$, mixed asymptotic setting \cite{cressie93}, that we call the mixed situation throughout).
  The paper is devoted to the estimation of the scale parameter $C$ from  one or several  generalized  quadratic $a$-variations $V_{a,n}$
defined in \eqref{def:V_A}. Calculations show that the expectation of $ V_{a,n}$ is a function of $C$ so that $C$ can be estimated by the moment method.

Our study is related to the study of  Istas and Lang \cite{IL97} in which they estimate both the scale parameter $C$ and the local Hölder index (a function of $D$ and $s$ in \eqref{eq:expansion}). Our main motivation for focusing on the case where the local Hölder index is known is, on the one hand, to provide a simpler method to implement and analyze the estimator, and on the other hand to address more advanced statistical issues, such as efficiency and aggregation of several estimators of $C$. In addition, our results hold under milder technical conditions than in \cite{IL97}, and in particular apply to most semi-variogram models commonly used in practice. In particular, we also show that a necessary condition in \cite{IL97}, namely the fact that the quantity in \eqref{eq:sumneq0} is non-zero when the variation used has a large enough order, in fact always holds. Thus, our study has a larger scope of application, in terms of necessary technical conditions, than that in \cite{IL97}.

We establish asymptotic approximations of the expectation and the variance and a central limit theorem for the quadratic variations  under consideration and for the estimators deduced from them. In particular, given a finite number of sequences $a$, we prove a joint central limit theorem (see Corollary \ref{cor:CLT_Can_joint}). In addition,  our method does not require a parametric  specification of the drift (see Section \ref{s:drift}); therefore it  is  more robust than MLE. 

For a finite discrete sequence  $a$ with zero sum, we define its order  as the largest integer $M$  such that 
$$
\sum_k a_k k^{\ell} =0 \quad \mbox{ for } \ell = 1,\ldots, M-1.
$$
Roughly speaking $ \sum_k a_k  f(k \delta_n)$  is an estimation of the $M$th derivative  of the  function $f$ at zero. The order of the simplest sequence: $-1,1$ is $M=1$. 
Natural questions then arise. What is the optimal  sequence  $a$?
 In particular, what is the optimal order? Is it better to use the elementary sequence of order 1 
 $(-1,1)$ or the one of order 2 $(-1,2,-1)$? For a given order, for example $M=1$, is it better to use  the elementary sequence of order 1 $(-1,1) $ or a more general one, for example $( -1,-2,3)$  or even a sequence based on discrete wavelets? 
   Can we efficiently  combine the information of several quadratic $a$-variations associated to several sequences?
 As far as we know, these questions are not addressed yet in the literature.  
Unfortunately, the asymptotic variance we give in Proposition \ref{prop:Van_Dqqe} or Theorem \ref{th:CLT_Can} does not allow either to address theoretically this issue. However, by Corollary \ref{cor:CLT_Can_joint}, one may gather the information of different quadratic $a$-variations with different orders. 
In order to validate such a procedure, an important Monte Carlo study is performed. The  main conclusion is that gathering the information of different quadratic $a$-variations with different orders $M$
produces closer results to  the  optimal  Cramér-Rao bound computed in Section \ref{section:cramer:rao}. The simulations are illustrated in Figure \ref{fig:Dzero:aggreg}. We also illustrate numerically the convergence to the asymptotic distribution considering different models (exponential and Matérn models).

Finally, we show that our suggested quadratic variation estimator can be easily extended to the two-dimensional case and we consider two real data sets in dimension two. When comparing our suggested estimator with maximum likelihood estimation, we observe a very significant computational benefit for our estimator.

\paragraph{Organization of the paper}
The paper is organized as follows. In Section \ref{sec:ass}, we detail the framework and present the assumptions on the process.  In Section \ref{sec:quad_a_var}, we introduce our quadratic variation estimator and provide its asymptotic properties, together with discussion.
Section \ref{section:opti} is devoted to the analysis of the statistical efficiency of our estimator.
In Section \ref{sec:num}, we provide the results of the Monte Carlo simulation and on the real data sets. 
A conclusion is provided in Section \ref{section:conclusion} together with some perspectives. 
 All the proofs have been postponed to the Appendix.

\section{General setting and assumptions}\label{sec:ass}

\subsection{Assumptions on the process}\label{ssec:ass:proc}

In this paper, we consider a Gaussian process $(X(t))_{t\in \R}$ which is not necessarily stationary but only has stationary increments. The process is observed at times $j\delta_n$ for $j=0,\dots,n$ with $\delta_n$ going to $0$ as $n$ goes to infinity. 
As mentioned in the introduction, we will only consider $\delta_n = n^{-\alpha}$  with $0 < \alpha \leq 1$  throughout the article.
Two cases are then considered: $\alpha = 1$ ($\delta_n = 1/n$, infill situation) and $0<\alpha<1$ ($\delta_n \to 0$ and $n \delta_n \to \infty$, mixed situation).
The semi-variogram of $X$ is defined by 
\begin{align*}
V(h):=\frac12\E\left[\left(X(t+h)-X(t)\right)^2\right].
\end{align*}
In the sequel, we denote by $(Const)$ a positive constant which value may change from one occurrence to another. 
For the moment, we assume that $X$ is centered, the case  of non-zero expectation will be considered  in Section  \ref{s:drift}. Now, we introduce  the following assumptions. The form of $\left(\mathcal{H}_{1}\right)$ and $\left(\mathcal{H}_{2}\right)$  change following whether we are in infill situation  or  in the particular mixed situation ($\delta_n=n^{-\alpha}$ with $0<\alpha< 1$).

\medskip

$\left(\mathcal{H}_{0}\right)$
 $V$ is a $C^{\infty}$ function on $(0,+\infty]$.

\medskip
{\bf  Infill situation:} $\delta_n = 1/n$. \\
$\left(\mathcal{H}_{1}\right)$ The semi-variogram  is $2D$ times differentiable with $D \geq 0$ and there exists $C>0$ and $0< s <2$ such that for any $h\in \R$, we have
\begin{equation} \label{e:cov}
V^{(2D)}(h)=   V^{(2D)}(0)  +C  (-1)^{D}  \abs{h}^{s}+r(h), \textrm{ with } \ r(h)=o(\abs{h}^s)  ~ ~   as \abs{h} \to 0.
\end{equation}

In $\left(\mathcal{H}_{1}\right)$, the integer $D$ is the greatest integer such that $V$ is $2D$-times differentiable everywhere.
We recall that, when $X$ is assumed to be a stationary process, we have $V(h)=k(0)-k(h)$. If the covariance function $k$ belongs to a parametric set of the form $\{  k_{\theta} ; \theta \in \Theta \}$ with $\Theta \subset \mathbb{R}^p$, then $C$ is a deterministic function of the parameter $\theta$.

$\left(\mathcal{H}_{2}\right)$ For some $\beta< -1/2$, for $\abs{h} <1$:
\begin{itemize} 
\item when $s<3/2$,
$$
\bigl| r^{(2)}(h)\bigr|  \leq (Const) \abs{h}^\beta;
$$
\item when $s\geq 3/2$,
$$
\bigl| r^{(3)}(h)\bigr|  \leq (Const) \abs{h}^\beta.
$$
\end{itemize}

$\left(\mathcal{H}_{3}\right)$ 
As $h\to 0$, 
$$
\abs{r(h)} = o \bigl( \abs{h}^{s+1/2}  \bigr).
$$

  \medskip

{\bf Mixed situation :} $\delta_n=n^{-\alpha}$ with $0<\alpha<1$.

We must add to $\left(\mathcal{H}_{1}\right)$: 
$$
 \abs{r(h)} \leq (Const) \abs{h}^s \quad \abs{h} >1.
 $$

The new expression  of $\left(\mathcal{H}_{2}\right)$ is 

\begin{itemize} 
\item when $s<3/2$, there exists $\beta$ with $s-2<\beta<-1/2$ such that,  for all $h \in \mathbb{R}$,
$$
\bigl| r^{(2)}(h)\bigr|  \leq (Const) \abs{h}^\beta;
$$
\item when $s\geq 3/2$, there exists $\beta$ with $s-3<\beta<-1/2$ such that,  for all $h \in \mathbb{R}$,
$$
\bigl| r^{(3)}(h)\bigr|  \leq (Const) \abs{h}^\beta.
$$
\end{itemize}
Here $\left(\mathcal{H}_{3}\right)$ writes
$$
\abs{r(h)} =  o \bigl( \abs{h}^{s+(1/2\alpha)}  \bigr),
$$
as $h\to 0$. 

\begin{rmk} \leavevmode
\begin{itemize}
\item When $D>0$, the $D$-th derivative $X^{(D)}$ in quadratic mean of $X$ is a Gaussian  stationary process with autocovariance function $k$ given by $k(h)=(-1)^{D+1}  V^{(2D)}(h)$. This implies that  the Hölder exponent of the paths of  $X^{(D)}$ is $s/2$. Because $s<2$,
$D$ is exactly the order  of differentiation  of  the paths of $X$. 
\item If we denote $H=D+s/2$, $H$ represents the local Hölder index of the process \cite{ibragimov78gaussian}.
\item Note that in the infill situation ($\delta_n=1/n$), $\left(\mathcal{H}_{2}\right)$ is almost minimal. Indeed, the condition $\beta<-1/2$ does not matter since the smaller $\beta$, the weaker the condition. And  for example, when $s<3/2$, the second derivative of the main term is of order $ \abs{h}^{s-2}$ and we only assume that $\beta>s-2$. 
\end{itemize}
\end{rmk}

\subsection{Examples of processes that satisfy our assumptions}\label{ssec:ex}

We present a non exhaustive list of examples in dimension one that satisfy our hypotheses. In these examples, we provide a stationary covariance function $k$, and we recall that this defines $V$, with $V(h) = k(0) - k(h)$. 
 \begin{itemize}
 \item  The exponential  model:  $k(h)=  \exp(- C |h|)$ ($D=0$, $s =1$, $C=C$). For this model, $\left(\mathcal{H}_{0}\right)$ to $\left(\mathcal{H}_{2}\right)$ always hold and $\left(\mathcal{H}_{3}\right)$ holds when $\alpha >1/2$, that is when the observation domain does not increase too fast.
 \item  The generalized exponential model:  $k(h) = \exp(- C |h|^s)$, $s \in(0,2)$ ($D=0$, $s =s$, $C=C$). For this model, $\left(\mathcal{H}_{0}\right)$ to $\left(\mathcal{H}_{2}\right)$ always hold and $\left(\mathcal{H}_{3}\right)$ holds when $1/(2\alpha) < s$. Hence, in the infill situation, we need $s >1/2$ and, in the mixed situation, the observation domain needs to increase slowly enough.
 \item The generalized Slepian model \cite{Slepian63}:  $k(h) =(1 - C |h|^s)^+, s \in(0,1]$ ($D=0$, $s =s$, $C=C$). For this model, $\left(\mathcal{H}_{0}\right)$ to $\left(\mathcal{H}_{3}\right)$ hold in the infill situation and when $C <1$. For $\left(\mathcal{H}_{0}\right)$, we remark that, in this case, $V$ is smooth on $(0,1]$ and not on $(0,\infty)$, but this is sufficient for all the results to hold.
 \item  The Matérn model:
\[
  k(h) =\frac{2^{1-\nu}
  }{\Gamma(\nu)} \big( \sqrt{2\nu} \theta h \big) ^\nu K_\nu ( \sqrt{2\nu} \theta h ),
\]
where $\nu>0$ is the regularity parameter of the process. The function $K_{\nu}$ is the modified Bessel function of the second kind of order $\nu$. See, e.g., \cite{stein99interpolation} for more details on the model. In that case, $D=\lfloor \nu\rfloor$ and $s = 2\nu-2D$. Here, it requires tedious computations to express the scale parameter $C$ as a function of $\nu$ and $\theta$. However, in Section \ref{ssec:simu}, we derive the value of $C$ in two settings ($\nu=3/2$ and $\nu=5/2$).
For this model, $\left(\mathcal{H}_{0}\right)$ to $\left(\mathcal{H}_{2}\right)$ always hold and $\left(\mathcal{H}_{3}\right)$ holds when $s <2-1/(2\alpha)$. Hence, in the infill situation we need $s <3/2$, and in the mixed situation the observation domain needs to increase slowly enough.
\end{itemize}
All the previous examples are stationary (and thus have stationary increments). The following one is not stationary.\begin{itemize}
 \item The fractional Brownian motion (FBM) process denoted by $(B_s(t))_{t\in \R}$ and  defined by 
\[
 \Cov (B_s(u),B_s(t)) =C \big(  |u|^s + |t|^s -|u-t|^s\big).
 \]
 A reference  on this subject is  \cite{CI13}.
   This process is classically indexed by its Hurst parameter $H= s/2$. 
Here, $D=0$, $s=s$ and $C=C$. We call the FBM defined by $C = 1$ the standard FBM.
\end{itemize}

\medskip

We remark that the Gaussian model, or square-exponential, defined by $k(h) = \sigma^2 e^{-h^2 \theta^2}$, with $(\sigma^2,\theta) \in (0,\infty)$, does not satisfy our assumptions, because it is too regular (i.e.\ it is $C^{\infty}$ everywhere).

 \paragraph{Detailed verification of the assumptions with the generalized exponential model}
 
We consider the generalized exponential model, where $k(h) = \exp(- C |h|^s)$ for some fixed $s \in(0,2)$.
Since we have $V(h) = 1 - \exp(- C |h|^s)$ for $h \in \mathbb{R}$, $\left(\mathcal{H}_{0}\right)$ is trivially satisfied. Now we show that $\left(\mathcal{H}_{1}\right)$ holds for $D=0$. Indeed, $V$ is a continuous function and we have
\begin{align*}
V(h) & = 0 + C |h|^s  + 1 -  C |h|^s - \exp(- C |h|^s) \\
& \defeq V(0) + C(-1)^0 |h|^s  + r(h),
\end{align*}
with 
\[
r(h) = 1 -  C |h|^s - \exp(- C |h|^s). 
\]
As $h \to 0$, $r(h) = O(|h|^{2s}) = o (|h|^s)$ and thus $\left(\mathcal{H}_{1}\right)$ holds in the infill situation. Furthermore, as $|h| \to \infty$,  $r(h) = -C |h|^s + o(1) = O( |h|^s )$ and so  $\left(\mathcal{H}_{1}\right)$ holds also in the mixed situation.
Let us now show that $\left(\mathcal{H}_{2}\right)$ is also satisfied.   
First consider the case where $s< 3/2$ and let
 \[
 \beta = 
 \begin{cases}
 \frac{1}{2} \left( (s-2) + (-1/2)  \right)
 & \text{if $2s-2 > -1/2 \quad (s > 3/4)$,} \\
\frac{1}{2} \left( (s-2) + (2s-2) \right)
 & \text{if  $2s-2 \leqslant -1/2 \quad (s \leqslant 3/4)$}.
 \end{cases}
 \] 
 One can check that $s-2 < \beta $ and that $\beta < -1/2$ since $s-2 < -1/2$. 
Consider the case where $|h| \leq 1$. Since
\[
r(h) = \sum_{k=2}^{\infty} 
(-1)^{k+1} \frac{C^k}{k!}
|h|^{sk},
\]   
one has
\begin{align*}
r^{(2)}(h)
& = 
|h|^{2s-2}
\sum_{k=2}^{\infty} 
(-1)^{k+1} \frac{C^k}{k!}
sk(sk-1)
|h|^{s(k-2)}  = |h|^{2s-2} g(h)
\end{align*}
where $g$ is a bounded function on $[-1,1]$. Hence, $|r^{(2)}(h)|  \leq (Const) |h|^{\beta}$ since $2s-2 \geq \beta$. Then the case $s< 3/2$ is complete in the infill situation.
 Consider now the case where $|h| \geq 1$. Simply, one can show that
 \[
 r^{(2)}(h) =  -C s (s-1) |h|^{s-2} + \left( -Cs(s-1)|h|^{s-2} 
+ C^2 s^2 |h|^{2s-2} 
 \right) \exp( -C|h|^s ).
 \]
 Hence 
 \[
  |r^{(2)}(h)| \leq (Const) |h|^{s-2} \leq (Const) |h|^{\beta} 
\quad \text{since $\beta > s-2$}. 
  \] 

The case where $s \geq 3/2$ can be treated analogously.
Finally, it is simple to show that $\left(\mathcal{H}_{3}\right)$ holds when $1/(2\alpha) < s$.

\subsection{Discrete \texorpdfstring{$a$}{f}-differences}\label{ssec:res_a_var}

Now, we consider a non-zero finite support sequence $a$ of real numbers with zero sum.  Let $L(a)$ be its length. Since the starting point of the sequence plays no particular role, we will assume when possible that  the first non-zero element is  $a_0$. Hence, the last  non-zero element  is $a_{L(a)-1}$. 
We define  the order $M(a)$ of the sequence as the first non-zero moment of the sequence $a$:
\begin{align*}
\sum_{j =0} ^{L(a) -1}  a_j j^k=0, \quad \textrm{for} \quad 0\leq k<M(a) \quad \textrm{and} \quad \sum_{j =0} ^{L(a) -1} a_j j^{M(a)}\neq 0.
\end{align*}

To any sequence $a$,  with length $L(a)$ and any function $f$, we associate the discrete $a$-difference of $f$ defined by
\begin{align}\label{def:delta_a}
\Delta_{a,i}(f)= \sum_{j=0} ^{L(a)-1} a_j f((i+j)\delta_n),\quad i=1,\dots n',
\end{align}
where $n'$ stands for $n-L(a)+1$.  As a matter of fact, in the case of the simple quadratic $a$-variation given by $a_0=0$ and $a_1=-1$, the operator $\Delta_a$ is a discrete differentiation operator of order one. More generally,
$\sum_{j=0} ^{L(a) -1} a_j f(j\delta_n)$ is an approximation (up to some multiplicative coefficient) of the $M(a)$-th derivative  (when it exists) of the function $f$ at zero. 

\medskip

We also define  $\mathbf{\Delta_{a}}(X)$  as the Gaussian vector of size $n'$ with entries  $\Delta_{a,i}(X) $ and $\Sigma_{a}$ its variance-covariance matrix. 

\medskip

\tbf{Examples - Elementary sequences}. The simplest case  is  the  order 1 elementary sequence $a^{(1)}$  defined by  $a_0^{(1)}=-1$ and $a_1^{(1)}=1$  We have  $L(a^{(1)})=2$,  $M(a^{(1)})=1$. More  generally, 
we define  the $k$-th order elementary sequence $a^{(k)}$  as the sequence  with  coefficients $a_j^{(k)}=(-1)^{k-j}\binom{k}{j}$, $j=0,\dots,k$. Its length is given by $L(a^{(k)})=k+1$. 

\medskip

For two sequences $a$ and $a'$, we define their convolution $b=a*a'$  as the sequence given by $b_j=\sum_{k-l=j} a_ka'_l$. In particular, we denote by $a^{2*}$ the convolution $a*a$. Notice that the first non-zero element of $b$ is not necessarily $b_0$ but $b_{L(a')-1}$ as mentioned in the following properties.


\begin{proptes} The following properties of convolution of sequences are direct. \leavevmode 
\begin{enumerate}
\item[(i)]  The support of $a*a'$ (the indices of the non-zero elements)  is included in $-(L(a')-1), (L(a)-1)$ while its order  is $M(a) +M(a')$. In particular, $a^{2*}$ has length $2L(a)-1$, order $2M(a)$ and is symmetrical.
\item[(ii)] The composition of two elementary sequences gives another elementary sequence.
\end{enumerate}
\end{proptes}

The main result of this section is Proposition \ref{prop:sumneq0} that is required to quantify the asymptotic behaviors of the two first moments of the quadratic $a$-variations defined in \eqref{def:Van} (see Proposition \ref{prop:Van_Dqqe}). In order to prove \eqref{eq:sumneq0}, we establish two preliminary tools (Proposition \ref{prop:ifbm} and Lemma \ref{l:jma}). 
In that view,  we need to define the integrated fractional Brownian motion (IFBM). We start from the FBM defined in Section \ref{ssec:ex} which has 
 the following  non anticipative  representation:
\[
   B_s(u) = \int_{-\infty} ^u f_s(t,u) dW(t),
\]
where $dW(t)$ is a white noise defined on the whole real line  and 
\[
   f_s(t,u)  = (Const)  \big( ((u-t)^+)^{(s-1)/2} - ((-t)^+)^{(s-1)/2} \big).
\]
   
For $m\geq0$ and $t\geq 0$, we define inductively the IFBM  by
\begin{align*}
B_s^{(-0)}(u) &= B_s(u)\\
B_s^{(-m)}(u) &= \int_0^u  B_s^{(-(m-1))}(t) dt.
\end{align*}
          
\begin{definition} [Non degenerated property]
A process $Z$ has the ND property if for every $k>0$ and every $ t_1<t_2 <\dots < t_k$  belonging to the domain of definition of  $Z$, the distribution of 
$
Z (t_1),\dots,Z(t_k)   $  is non degenerated.
\end{definition}

 We have the following results.

 \begin{prop} \label{prop:ifbm}
 The IFBM has the ND property.
 \end{prop}

 \begin{lemma} \label{l:jma} 
 The variance function  of the IFBM  satisfies, for all $m \in \mathbb{N}$,
\[
 \Var  \big( B_s^{(-m)}(u) -  B_s^{(-m)}(v) \big) =  
 \sum_{i=1}^{N_m} \left(
 P^{m,i}(v) h_{m,i}(u)
  + 
 P^{m,i}(u) h_{m,i}(v) \right)
 + (-1)^m \frac{ 2|u-v|^{s+2m}}{
 (s+1)\dots (s+2m)},
\]
 where $N_m \in \mathbb{N}$, for $i=1,...,N_m$, $P^{m,i}$ is a polynomial of degree less or equal to $m$ and $h_{m,i}$ is some function.
 \end{lemma}

 \begin{prop}\label{prop:sumneq0} If the sequence $a$ has order $M(a)>D$, then
\begin{align}\label{eq:sumneq0}
\sum_{j}a_j^{2*} \abs{j}^{2D+s} \neq 0 \quad (\trm{i.e.} \quad (-1)^D\sum_{j}a_j^{2*} \abs{i}^{2D+s} < 0).
\end{align}
\end{prop}

Note that  \eqref{eq:sumneq0} is stated as  an hypothesis in \cite{IL97}.

\section{Quadratic \texorpdfstring{$a$}{f}-variations}\label{sec:quad_a_var}

\subsection{Definition}

Here, we consider the discrete $a$-difference applied to the process $X$ and we define the quadratic $a$-variations by
\begin{align}\label{def:Van}
V_{a,n} = \| \mathbf{\Delta_a}(X)\|^2 =\sum_{i=1} ^{n'}(\Delta_{a,i}(X))^2,
\end{align}
recalling that $n'=n-L(a)+1$. When no confusion is possible, we will use the shorthand notation $L$ and $M$ for $L(a)$ and $M(a)$.

\subsection{Main results on quadratic \texorpdfstring{$a$}{f}-variations}\label{ssec:results_main}

The basis of our  computations of variances is the identity
\begin{align}\label{eq:prop3_aa}
\E[\Delta_{a,i}(X)\Delta_{a',i'}(X)] = -\Delta_{a*a',i-i'}(V),
\end{align}
for any sequences $a$ and $a'$.
A second main tool is the  Taylor expansion with integral remainder (see, for example, \eqref{e:zaza2}). So we introduce another notation. For a sequence $a$, a scale $\delta$, an order $q$ and a function $f$, we define 
\begin{align}\label{def:R}
R(i,\delta,q,f,a) 
& = - \sum_{j} a_j j^q \int_0^1 \frac{(1-\eta)^{q-1}}{(q-1)!} f((i+j\eta)\delta)d\eta.
\end{align}
By convention, we let $R(i,\delta,0,f,a)= -\Delta_{a,i}(f)$. Note that $R(-i,\delta,2q,\abs{\cdot{}}^s,a*a')=R(i,\delta,2q,\abs{\cdot{}}^s,a'*a)$.
One of our main results is the following. 

\begin{prop}[Moments of $V_{a,n}$]\label{prop:Van_Dqqe}
Assume that $V$ satisfies $\left(\mathcal{H}_{0}\right)$ and $\left(\mathcal{H}_{1}\right)$. 

{\bf 1)} If we choose a sequence $a$ such that $M>D$, then
\begin{align}\label{eq:esp_van_Dqqe}
\E[V_{a,n}] = n C (-1)^D \delta_n^{2D+s} \left[R(0,1,2D,\abs{\cdot}^s,a^{2*})\right](1+o(1)),
\end{align}
as $n$ tends  to infinity. Furthermore, $ (-1)^D R(0,1,2D,\abs{\cdot}^s,a^{2*})$ is positive. 

\medskip

{\bf 2)} If $V$ satisfies additionally  $\left(\mathcal{H}_{2}\right)$  and if we choose a sequence $a$ so that $M>D+s/2+1/4$, then as $n$ tends  to infinity:
\begin{align}\label{eq:var_van_Dqqe}
\Var(V_{a,n}) = 2n C^2\delta_n^{4D+2s}  
\sum_{i\in \Z}  R^2(i,1,2D,\abs{\cdot{}}^{s},a^{2*})  (1+o(1))
\end{align}
 and the series above is positive and finite. 
\end{prop}

\begin{rmk} (i) Notice that \eqref{eq:esp_van_Dqqe} and \eqref{eq:var_van_Dqqe} imply concentration in the sense that
\[
\frac{V_{a,n}}{\E[V_{a,n}]} \underset{n \to +\infty}{\overset{L^2}{\longrightarrow}}  1.
\]

(ii) In practice, since the parameters $ D$ and $ s$ are known, it suffices to  choose $M$ such that   $M\geq D+1$ when $s<3/2$ and $M\geq D+2$ when $3/2 \leq s<2$.

(iii) The expression of the asymptotic variance appears to be complicated. Anyway, in practice, it can be easily approximated. Some explicit examples are given in Section \ref{sec:num}.
\end{rmk}



Following the same lines as in the proof of Proposition \ref{prop:Van_Dqqe} and using the identities $(a*a')_j=(a'*a)_{-j}$ and $R(i,1,2D,\abs{\cdot{}}^{s},a*a')=R(-i,1,2D,\abs{\cdot{}}^{s},a'*a)$, one may easily derive the corollary below. The proof is omitted.

\begin{cor}[Covariance of $V_{a,n}$ and $V_{a',n}$]\label{cor:Van}
Assume that $V$ satisfies $\left(\mathcal{H}_{0}\right)$, $\left(\mathcal{H}_{1}\right)$,  and $\left(\mathcal{H}_{2}\right)$. Let us consider two  sequences $a$ and $a'$ so that $M(a)\wedge M(a')>D +s/2+1/4$. 
Then,  as $n$ tends  to infinity, one has
\begin{align}\label{eq:cov_van}
\Cov(V_{a,n},V_{a',n}) =   2n C^2\delta_n^{4D+2s}\left[\sum_{i\in \Z} 
 R^2(i,1,2D,\abs{\cdot{}}^{s},a*a')\right] (1+o(1)).
\end{align}
\end{cor}

\tbf{Particular case - $D=0$}:
\begin{enumerate}
\item[(i)] We choose $a$ as the first order elementary sequence ($a_0=-1$, $a_1=1$ and $M=1$). As $n$ tends to infinity, one has 
\begin{align*}
\E[V_{a,n}] & = n C \delta_n^{s} (2+o(1));\\
\Var(V_{a,n}) & = 2n C^2\delta_n^{2s}  \sum_{i\in \Z} \left(\abs{i-1}^s-2\abs{i}^s+\abs{i+1}^s\right)^2 (1+o(1)),\; s<3/2.
\end{align*}
\item[(ii)] General sequences. We choose two sequences $a$ and $a'$ so that 
$M(a)\wedge M(a')> s/2+1/4$. Then, as $n$  tends to infinity, one has
\begin{align*}
\E[V_{a,n}] & = - n C \delta_n^{s}\left[\sum_{j} a_j^{2*} \abs{j}^s\right](1+o(1));\\
\Var(V_{a,n}) & = 2n C^2\delta_n^{2s}  \sum_{i\in \Z} \left(\sum_j  a_j^{2*} \abs{i+j}^s \right)^2(1+o(1));\\
\Cov(V_{a,n},V_{a',n}) & = 2n C^2\delta_n^{2s} \left(\sum_{\abs{j}\leq L}  a*a'_j \abs{j}^s \right)^2 (1+o(1))\\
& +n C^2\delta_n^{2s} \sum_{i\in \Z^*}\left( \left(\sum_{\abs{j}\leq L}  a*a'_j \abs{i+j}^s \right)^2 + \left(\sum_{\abs{j}\leq L}  a'*a_j \abs{i+j}^s \right)^2\right) (1+o(1)).
\end{align*}
\end{enumerate}

Now we establish the  central limit theorem.

\begin{theorem}[Central limit theorem for $V_{a,n}$] \label{th:CLT_Van} Assume $\left(\mathcal{H}_{0}\right)$, $\left(\mathcal{H}_{1}\right)$ and $\left(\mathcal{H}_{2}\right)$ and $M> D+s/2+1/4$. Then $V_{a,n}$ is asymptotically normal in the sense that
\begin{align} \label{e:jma}
\frac{ V_{a,n} -\E[V_{a,n}]}{\sqrt{\Var(V_{a,n})}} \underset{n \to +\infty}{\overset{D}{\longrightarrow}}  \mathcal{N}(0,1).
\end{align}
\end{theorem} \medskip

\begin{rmk}\label{rem:cas_pourris}
\leavevmode
\begin{itemize}
\item If $M=D+1$, the condition $M>D+s/2+1/4$ in Proposition \ref{prop:Van_Dqqe} implies $s<3/2$. However, when $M = D+1$ and $s\geq 3/2$, it is still possible to  compute the variance but the convergence is slower and the central limit theorem does not hold anymore. More precisely, we have the following.
 \begin{itemize}
  \item If $s>3/2$  and $M = D+1$ then,  as $n$ tends to infinity,
  \begin{align}\label{eq:var_van_dv_Dqqe}
\Var(V_{a,n})  = (Const)\times \delta_n^{4D+2s} \times n^{2s-4(M-D)+2} \times (1+o(1)).
\end{align}

 \item If $s=3/2$ and $M = D+1$ then, as  $n$ tends to infinity
 \begin{align}\label{eq:var_van_dv_pb_Dqqe}
\Var(V_{a,n})= (Const)\times  \delta_n^{4D+2s} \times n\log n \times(1+o(1)).  
\end{align}
\end{itemize}
We omit the proof. Analogous formula for the covariance of two variations can be derived similarly. 
\item Since the work of  Guyon and Le\' on  \cite{GL89}, it is a well known fact that  in the simplest  case ($D=0,L=2,M=1)$ and in the infill situation ($\delta_n=1/n$, $\alpha =1$),  the  central limit theorem holds true for quadratic variations if and only if $s<3/2$. Hence assumption $M>D+s/2+1/4$ is minimal.
\end{itemize}
\end{rmk}

%

\begin{cor}[Joint central limit theorem] \label{cor:CLT_Van_joint}
Assume that $V$ satisfies $\left(\mathcal{H}_{0}\right)$, $\left(\mathcal{H}_{1}\right)$ and $\left(\mathcal{H}_{2}\right)$. Let $a^{(1)},\dots,a^{(k)}$  be $k$ sequences with order greater than $D+s/2+1/4$.
  Assume also that, as $n \to \infty$, the $k \times k$ matrix with term $i,j$ equal to
 \[
 \frac{1}{n \delta_n^{4D+2s}}
 \Cov \left( V_{a^{(i)},n} , V_{a^{(j)},n} \right)
 \]
converges to an invertible matrix $\Lambda_{\infty}$. Then, $V_{a^{(1)},\dots,a^{(k)},n} = ( V_{a^{(1)},n},\dots,V_{a^{(k)},n} )^\top$ is asymptotically normal in the sense that  $n \to \infty$
\[
\frac {V_{a^{(1)},\dots,a^{(k)},n} -
\mathbb{E}
\left[
V_{a^{(1)},\dots,a^{(k)},n}
\right] }
{n^{1/2}\delta_n^{2D+s}}
\underset{n \to +\infty}{\overset{D}{\longrightarrow}} 
\mathcal{N}( 0 , \Lambda_\infty).
\]
\end{cor}

\subsection{Estimators of C based on the quadratic a-variations}\label{ssec:Can}

Guided by the moment method, we define
\begin{align}\label{eq:Can}
C_{a,n} := \frac{ V_{a,n}}{ n (-1)^D \delta_n^{2D+s} R(0,1,2D,\abs{\cdot}^s,a^{2*}) }.
\end{align}
Then $C_{a,n}$ is an estimator of $C$ which is asymptotically unbiased
by Proposition \ref{prop:Van_Dqqe}. Now our aim is to establish its asymptotic behavior.

\begin{theorem}[Central limit theorem for $C_{a,n}$] \label{th:CLT_Can}
Assume $\left(\mathcal{H}_{0}\right)$  to $\left(\mathcal{H}_{3}\right)$ and that  $M(a) > D +s/2 + 1/4$. 
Then
$C_{a,n}$ is asymptotically normal. More precisely, we have
\begin{align}\label{eq:CLT_Can}
\frac{C_{a,n} -C}{\sqrt{\Var(C_{a,n})}} \underset{n \to +\infty}{\overset{D}{\longrightarrow}}  \mathcal{N}(0,1),
\end{align}
with
$\Var(C_{a,n})=(Const) n^{-1} (1+o(1))$.
\end{theorem}

The following corollary is of particular interest: it will give theoretical results when one aggregates the information of different quadratic $a$-variations with different orders. As one can see numerically in Section \ref{ssec:simu_aggregation}, such a procedure appears to be really promising and circumvents the problem of the determination of the optimal sequence $a$.

\begin{cor}\label{cor:CLT_Can_joint}
Under the assumptions of Theorem \ref{th:CLT_Can}, consider $k$ sequences $a^{(1)},\ldots,a^{(k)}$ so that, for $i=1,\ldots,k$, $M(a^{(i)}) > D +s/2 + 1/4$. Assume furthermore that the covariance matrix of\\ $( C_{a^{(i)},n} /  \Var(C_{a^{(i)},n})^{1/2} )_{i=1,\ldots,k}$ converges to an invertible matrix $\Gamma_{\infty}$ as $n \to \infty$. Then, $( [C_{a^{(i)},n} - C] / \Var(C_{a^{(i)},n})^{1/2} )_{i=1,\ldots,k}$ converges in distribution to the $\mathcal{N}( 0 , \Gamma_{\infty})$ distribution.
\end{cor}

\subsection{Adding a drift} \label{s:drift}

In this section, we do not assume anymore that the process $X$ is centered  and we set  for $t \geq 0$,
\[
f(t) = \E [X(t)].
\]
We write $\overline X$ the corresponding centered process: $\overline X (t)=X(t) - f(t)$.
As it is always the case in statistical applications, we assume  that $f$ is a $C^{\infty}$ function. We emphasize on the fact that our purpose is not proposing an estimation of the mean function.

\begin{cor} \label{cor:cr}
Assume the same assumptions as in Theorem \ref{th:CLT_Can}, and recall that $\delta_n = n^{-\alpha}$ for $\alpha \in (0,1]$.
Define 
\[
K^\alpha_{M,n}  = \sup_{t \in [0, n^{1-\alpha}] } |f^{(M)}(t)|.
\]
 and if we assume in addition that
\begin{align}\label{eq:cond_K}
K^\alpha_{M,n} = o( n^{-1/4} \delta_n^{D-M+s/2}),
\end{align}
 then \eqref{eq:CLT_Can} still holds for $X$.
\end{cor}

Note that in the infill situation ($\delta_n=1/n$, $\alpha =1$), $K^1_{M,n}$ does not depend on $n$.  Obviously, \eqref{eq:cond_K} is met if $f$ is a polynomial up to an appropriate choice of the sequence $a$ (and $M$). In the infill situation, a sufficient condition 
for \eqref{eq:cond_K} is $M >D+ s/2+1/4$ which is always true. Moreover, it is worth noticing that we only assume regularity on the $M$-th derivative of the drift. No parametric assumption on the model is required, unlike in the MLE procedure.

\subsection{Elements of comparison with existing procedures}\label{ssec:compar}

\subsubsection{Quadratic variations versus MLE}

In this section, we compare our methodology to the very popular MLE method.
For details on the MLE procedure, the reader is referred to, e.g.\, \cite{rasmussen06gaussian,SWN03}.

\paragraph{Model flexibility} As mentioned in the introduction, the MLE methodology is a parametric method and requires the covariance function to belong to a parametric family of the form $\{k_{\theta} , \theta\in \Theta \}$. In the procedure proposed in this paper, it is only assumed that the semi-variogram satisfies the conditions given in Section \ref{ssec:ass:proc}, and that $D$ and $s$ are known. In this latter case, the suggested variation estimator is feasible, while the MLE is not defined.

\paragraph{Adding a drift}
In order to use the MLE estimator, it is necessary to assume that the mean function of the process is a linear combination of known parametric functions:
\begin{align*}
f(t)=\sum_{i=1}^q \beta_i f_i(t),
\end{align*}
with known $f_1,\ldots,f_q$ and where $\beta_1,\ldots,\beta_q$ need to be estimated. Our method is less restrictive and more robust. Indeed, we only assume the regularity of the $M$-th derivative of the mean function in assumption \eqref{eq:cond_K}. It does not require parametric assumptions neither on the semi-variogram nor the mean function. Furthermore, no estimation of the mean function is necessary.

\paragraph{Computational cost} The cost of our method is only $O(n)$ (the method only requires the computation of a sum) while the cost of the MLE procedure is known to be $O(n^3)$.

\paragraph{Practical issues} In some real data frameworks, it may occur that the MLE estimation diverges as can be seen in Section \ref{ssec:real:data}. Such a dead end can not be possible with our procedure.

\subsubsection{Quadratic variations versus other methods}

\paragraph{Least-square estimators}

In \cite{LLC02}, the authors propose a Least Square Estimator (LSE). More precisely, given an estimator $V_n$ of the semi-variagram assumed to belong to a parametric family of semi-variograms $\{V_{\theta} , \theta \in \Theta\}$, the Ordinary Least Square Estimation (OLSE) consists in minimizing in the parameter $\theta$ the quantity
\[
\sum_{h \in I} (V_n(h)-V_{\theta}(h))^2,
\]
where $I$ is a set of lags and $V_n$ is a non-parametric estimator of the semi-variogram.
Several variants as the Weighted Least Square Estimation and the General Least Square Estimation have been then introduced.
Then the authors of \cite{LLC02} provide necessary and sufficient conditions for these estimators to be asymptotically efficient and they show that when the number of lags used to define the estimators is chosen to be equal to the number of variogram parameters to be estimated, the ordinary least squares estimator, the weighted least squares and the generalized least squares
estimators are all asymptotically efficient. Similarly as for the MLE, least square estimators require a parametric family of variograms, while quadratic variation estimators do not.

\paragraph{Cross validation estimators}
Cross validation estimators \cite{Bac2014,bachoc2018asymptotic,bachoc2013cross,zhang2010kriging}
are based on minimizing scores based on the leave one out prediction errors, with respect to covariance parameters $\theta$, when a parametric family of semi-variograms $\{V_\theta ,  \theta \in \Theta\}$ is considered. Hence, as the MLE, they require a parametric family of variograms. Furthermore, as the MLE, the computation cost is in $O(n^3)$, while this cost is $O(n)$ for quadratic variation estimators. 

\paragraph{Composite likelihood}

Maximum composite likelihood estimators follow the principle of the MLE, with the aim of reducing its computational cost
\cite{Varin:Reid:Firth:2011,mateu2007fitting,pardo1997amle3d,stein2004approximating,vecchia1988estimation}.
In this aim, they consist in optimizing, over $\theta$, the sum, over $i=1,\ldots,n$ of the conditional likelihoods of the observation $X_i = X(t_i)$, given a small number of observations which observation locations are close to $t_i$
when a parametric family of semi-variograms $\{V_\theta ,  \theta \in \Theta\}$ is considered. The computation cost of an evaluation of this sum of conditional likelihood is $O(n)$, in contrast to $O(n^3)$ for the full likelihood. Nevertheless, the composite likelihood estimation requires to perform a numerical optimization, while our suggested estimator does not. Furthermore, a parametric family of variograms is required for the composite likelihood but not for our estimator. Finally, \cite{BL19} recently showed that the composite likelihood estimator has rate of convergence only $n^{s}$ when $D=0$ and $0 < s <1/2$ in \eqref{eq:expansion}. Hence, the rate of convergence of quadratic variation estimators ($n^{1/2}$) is larger in this case.

\subsubsection{Already known results on quadratic $a$-variations}

In \cite{IL97}, Istas and Lang consider a Gaussian process with stationary increments in the infill case and assume \ref{eq:expansion} as in our paper. Then they establish the asymptotic behavior of $V_{a,n}$ under more restrictive hypothesis of regularity on $V$ than ours (in particular on $r$ and on $\delta_n$). Then they propose an estimation of both the local Hölder index $H=D+s/2$ and the scale parameter $C$, based on quadratic $a$-variations and study their asymptotic behavior. The expression of the estimation of $C$ is much more complex than our that simply stems from the moment method. More precisely, they consider $I$ sequences $(a^j)_{j=1,\dots,I}$ with length $p_j$ and the vector $U$ of length $I$ whose coordinate $j$ is given by $V_{a^j,n}$. Noticing that the vector
$U/n$ converges to the product $AZ$ where $A$ is a $I\times p$ matrix derived from the sequences $a$ with $p=\max_j p_j$ and $Z$ is the vector of $(C(-1)^D(j\delta_n)^{2h})_{j=1,\dots,p}$, they estimate $Z$ by $\hat Z=(A^{\top} A)^{-1}A^\top U$ and derive their estimators of $h$ and $C$ from $\hat Z$.  

\medskip

As explained in the introduction in Section \ref{sec:intro}, Lang and Roueff in \cite{LR01} generalize the results of Istas and Lang in \cite{IL97} and \cite{KW97}. They consider the infill situation and 
use quadratic $a$-variations to estimate both the scale parameter $C$ and smoothness parameter $s$ under a similar hypothesis as in \ref{eq:expansion}. Furthermore, they assume three types of regularity assumptions on $V$: Hölder regularity of the derivatives at the origin, Besov regularity and global Hölder regularity. Nevertheless, estimating both $C$ and $s$ leads to a more complex estimator of $C$ and to proofs significantly different and more complicated.

\medskip

To summarize, our contributions, additionally to the existing references
\cite{LR01,IL97,KW97}, is to provide an estimation method for $C$ which definition, implementation and asymptotic analysis are simpler. As a result, we need fewer technical assumptions. In fact, our assumptions can be easily shown to hold in many classical examples. This also enables us to study the aggregation of quadratic variation estimators from different sequences, see Section \ref{ssec:simu_aggregation}.

\section{Efficiency of our estimation procedure} \label{section:opti}

In this section, in order to decrease the asymptotic variance, we propose a procedure to combine several quadratic $a$-variations leading to aggregated estimators. Then our goal is to evaluate the quality of these proposed estimators. In that view, we compare their asymptotic variance with the theoretical  Cramér-Rao bound in some particular cases in which this bound can be explicitly computed.

 \subsection{Aggregation of estimators} \label{ssec:aggregation}

Now in order to improve the estimation procedure, we suggest to aggregate a finite number of estimators:
\[
\sum_{j=1}^k \lambda_j C_{a^{(j)},n} 
\]
based on $k$ different sequences  $a^{(1)},...,a^{(k)}$ with weights $\lambda_1,\dots,\lambda_k$. Ideally, one should provide an adaptive statistical procedure to choose the optimal number $k^*$ of sequences, the optimal sequences and the optimal weights $\lambda^*$.  Such a task is beyond the scope of this paper. Nevertheless, in this section, 
we consider a given number $k$ of given sequences $a^{(1)},...,a^{(k)}$ leading to the estimators $C_{a^{(1)},n},\dots ,C_{a^{(k)},n}$ defined by \eqref{eq:Can}. 
Then we provide the optimal weights $\lambda^*$. Using  \cite{lavancier2016general} or \cite{bates1969combination}, one can establish the following lemma.

\begin{lemma}
We assume that for $j=1,...,k$, the conditions of Corollary  \ref{cor:CLT_Can_joint} are met. Let $R$ be the  $k \times k$  asymptotic variance-covariance matrix  of  the vector of length $k$ whose elements are given by $(n^{1/2}/C) C_{a^{(j)},n}$, $j=1,\dots,k$.  Then for any $\lambda_1,\dots,\lambda_k$, 
\[
(n^{1/2}/C)( \sum_{j=1}^k \lambda_j C_{a^{(j)},n} - C)\underset{n \to +\infty}{\overset{D}{\longrightarrow}}  \mathcal{N}(0,\lambda^{T} R \lambda).
\]

\medskip

Let $\mathbf{1}_{k}$ be  the "all one" column vector  of size $k$  and define
\[
\lambda^*
=
\frac{
R^{-1}
\mathbf{1}_{k}
}
{
\mathbf{1}_{k}^T
R^{-1}
\mathbf{1}_{k}
}.
\]
One has $\sum_{j=1}^k \lambda^*_j = 1$  and
\[
\lambda^{*T} R \lambda^*\leqslant \lambda^{T} R \lambda.
\]
\end{lemma}

As will be shown with simulations in Section \ref{sec:num}, the aggregated estimator considerably improves each of the original estimators $C_{a^{(1)},n},...,C_{a^{(k)},n}$. We call $\widetilde{v}_{a,s}$ its normalized asymptotic variance.

\subsection{Cramér-Rao bound} \label{section:cramer:rao}

To validate the aggregation procedure, we want to compare the obtained asymptotic variance with the theoretical  Cramér-Rao bound. In that view, we compute in the following section the Cramér-Rao bound in two particular cases. 

\medskip

We consider a family $Y_C$ ($C \in  \R^+$)  of centered Gaussian processes.
Let $R_C$ be the $(n-1) \times (n-1)$ variance-covariance matrix defined by
\[
(R_C)_{i,j}
=
\Cov
\left( 
Y_C 
\left( i \delta_n \right)
- 
Y_C 
\left( (i-1) \delta_n
\right)
,
Y_C 
\left( i \delta_n \right)
- 
Y_C
\left( (i-1)\delta_n
\right)
\right).
\] 
Assume that $C \mapsto R_C$ is twice differentiable and  $R_C$ is invertible for all $C \in \R^+$. Then, let 
\begin{equation} \label{e:zaza6}
I_C = 
\frac{1}{2} \mathrm{Tr} 
\left( 
R_C^{-1}
\left(
\frac{\partial}{\partial{C}}
R_C
\right)
R_C^{-1}
\left(
\frac{\partial}{ \partial{C}}
R_C
\right)
\right)
\end{equation}
be the Fisher information. The quantity $1/I_C$ is the Cramér-Rao lower bound for estimating $C$ based on 
\[
\big\{ Y_C
(  i \delta_n)
- 
Y_C 
((i-1)\delta_n) \big\}_{i=2,...,n}
\]
(see for instance \cite{Bac2014,dahlhaus1989efficient}).  Now we give two examples of  families of processes for which we can compute the Cramér-Rao lower bound explicitly. The first example is obtained from the IFBM defined in Section \ref{ssec:ex}.

\begin{lemma}  \label{lem:cramer:rao:fbm}
Let $0 < s < 2$ and let $X$  be  equal to $ \sqrt{C}B^{(-D)}_s$ where $B^{(-D)}_s$ is the IFBM. Then  $ Y_C = X^{(D)}$ is a FBM whose semi-variogram $V_C$ is given by 
\begin{equation} \label{e:zaza5}
V_C(h)=\frac{1}{2}
\mathbb{E} \left[
\left(
Y_C(t+h)
-
Y_C(t)
\right)^2
\right]
=
C |h|^s.
\end{equation}
Hence in this case, we have $1/I_C=2C^2/(n-1)$.
\end{lemma}
\begin{proof}

\eqref{e:zaza5}  implies that  $\partial R_C/ \partial{C}
 = R_1$  then \eqref{e:zaza6}  gives the result. 
\end{proof}

Now we consider a second example given by the generalized  Slepian process defined in Section \ref{ssec:ex}. 

Let $s\leq 1$ and $ Y_C$ with stationary autocovariance function $k_C$ defined by
\begin{equation} \label{eq:def:gen:slepian}
 k_C(h) =(1-(C/2)|h|^s)^+, \quad \textrm{for any} \quad h \in \mathbb{R}.
\end{equation}
 This function is convex on $\mathbb{R}$ and  it follows from P\'olya's theorem \cite{polya1949remarks} that $k_C$ is a  valid autocovariance function. We thus easily obtain the following lemma whose proof is omitted.
 \begin{lemma}  \label{lem:cramer:rao:slepian}
 Let $X$ be the integration $D$ times  of $Y_C$ defined via \eqref{eq:def:gen:slepian}. Then, in the infill situation ($\delta_n=1/n$, $\alpha =1$)  and  for $C <2$, the semi-variogram of $Y_C$ is given by \eqref{e:zaza5} and by consequence $1/I_C = 2 C^2/(n-1)$. 
 \end{lemma}

\section{Numerical results}\label{sec:num}

In this section, we first study to which extent the asymptotic results of Proposition \ref{prop:Van_Dqqe} and Theorem \ref{th:CLT_Can} are representative of the finite sample behavior of quadratic $a$-variations estimators. Then, we study the asymptotic variances of these estimators provided by Proposition \ref{prop:Van_Dqqe} and that of the aggregated $a$-variations estimators of Section \ref{ssec:aggregation}.

\subsection{Simulation study of the convergence to the asymptotic distribution}\label{ssec:simu}

We carry out a Monte Carlo study of the quadratic $a$-variations estimators in three different cases. In each of the three cases, we simulate $N = 10,000$ realizations of a Gaussian process on $[0,1]$ with zero mean function and stationary autocovariance function $k$.
In the case $D=0$, we let $k(h) = \exp(-C|h|)$. Hence $\left(\mathcal{H}_{1}\right)$ holds with $D=0$ and $s=1$. In the case $D=1$, we use the  Mat\'ern $3/2$ autocovariance
\cite{roustant12dice} :
  \[
k(h) = 
\left( 1 + \sqrt{3} \frac{|h|}{\theta} \right)
e^{ - \sqrt{3} \frac{|h|}{\theta} }.
\]
 One can show, by developing $k$ into power series, that $\left(\mathcal{H}_{1}\right)$ holds with $D=1$, $s=1$ and $C = 6 \sqrt{3} / \theta^3$. Finally, in the case $D=2$, we use the Mat\'ern 5/2 autocovariance function: 
\[
k(h) = 
\left( 1 + \sqrt{5} \frac{|h|}{\theta} +  \frac{5 |h|^2}{3 \theta^2} \right)
e^{ - \sqrt{5} \frac{|h|}{\theta} }.
\]
Also $\left(\mathcal{H}_{1}\right)$ holds  true with $D=2$, $s=1$ and $C = 200  \sqrt{5} / 3 \theta^5$.  \bigskip

In each of the three cases, we set $C=3$. For $n=50$, $n=100$ and $n=200$, we observe each generated process at $n$ equispaced observation points on $[0,1]$ and compute the quadratic $a$-variations estimator $C_{a,n}$ of Section \ref{ssec:Can}. When $D=i$, $i=0,1,2$, we choose $a$ to be the elementary sequence of order $i+1$. 

\medskip

 In Figure \ref{fig:histograms}, we display the histograms of the $10,000$ estimated values of $C$ for the nine configurations of $D$ and $n$. We also display the corresponding asymptotic Gaussian probability density functions provided by Proposition \ref{prop:Van_Dqqe} and Theorem \ref{th:CLT_Can}.  We observe that there are few differences between the histograms and limit probability density functions between the cases ($D=0,1,2$). In these three cases, the limiting Gaussian distribution is already a reasonable approximation when $n=50$. This approximation then improves for $n=100$ and becomes very accurate when $n=200$. Naturally, we can also see the estimators' variances decrease as $n$ increases. Finally, the figures suggest that the discrepancies between the finite sample and asymptotic distributions are slightly more pronounced with respect to the difference in mean values than to the difference in variances. As already pointed out, these discrepancies are mild in all the configurations.

\begin{figure}

\begin{tabular}{ccc}
\includegraphics[width=5cm]{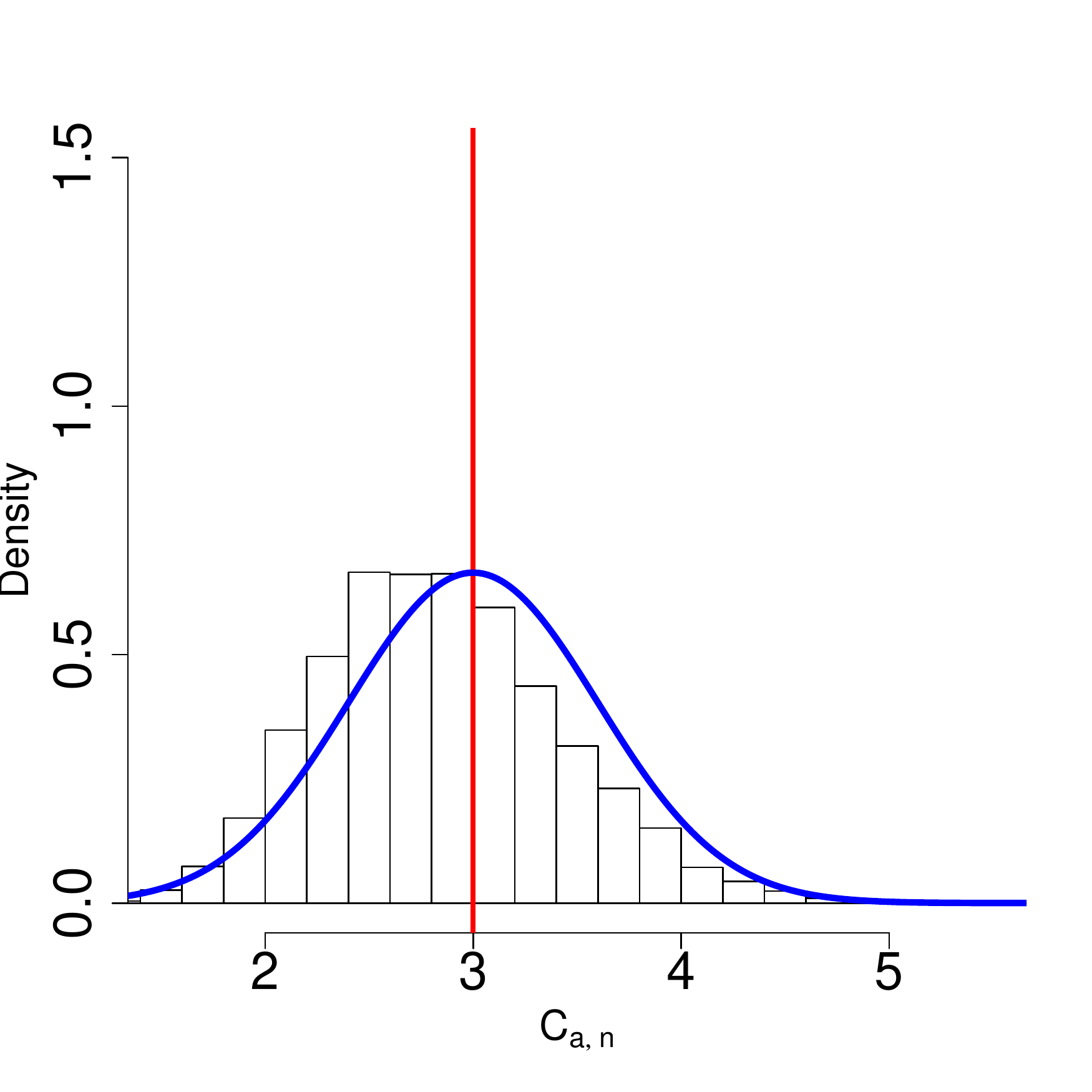}
&
\includegraphics[width=5cm]{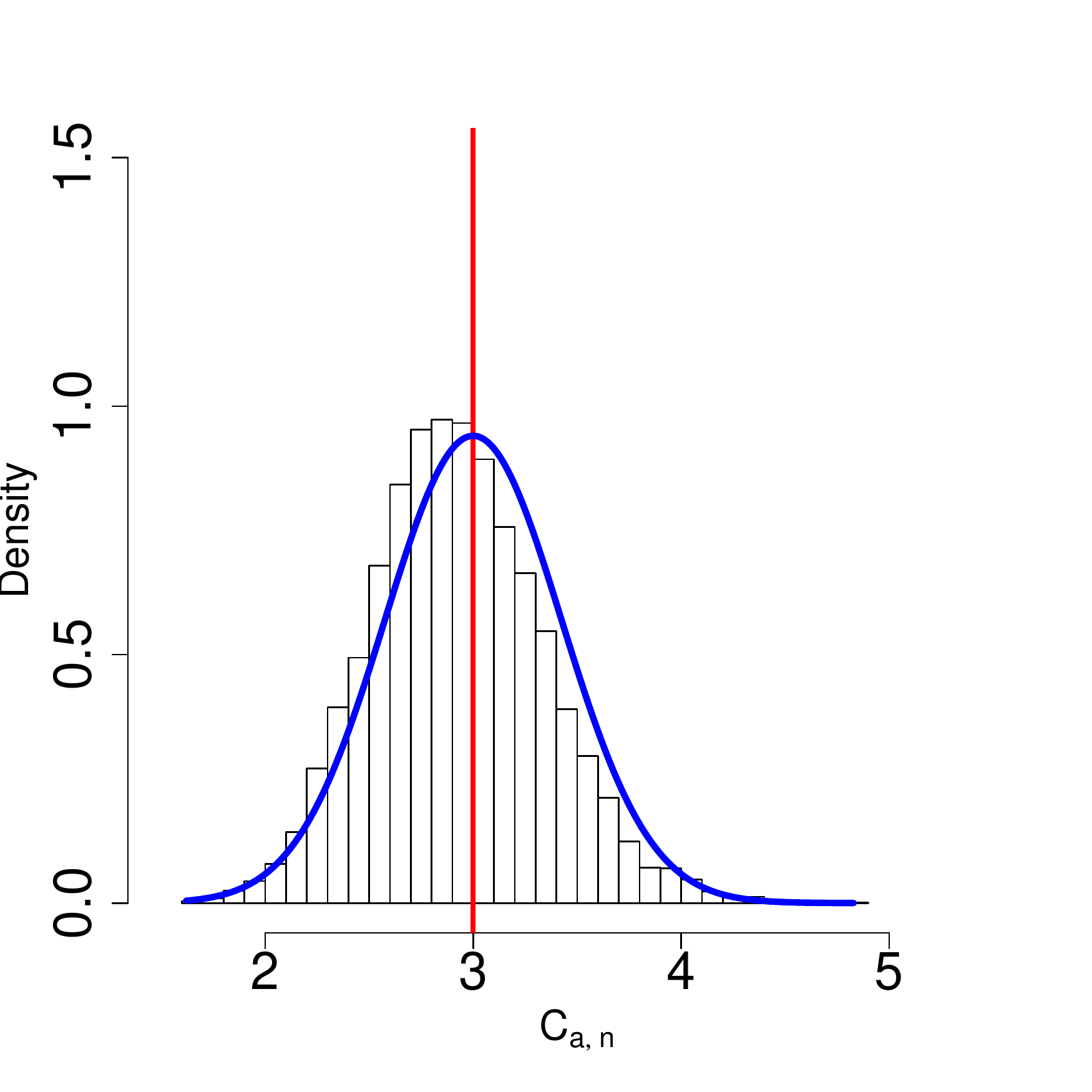}
&
\includegraphics[width=5cm]{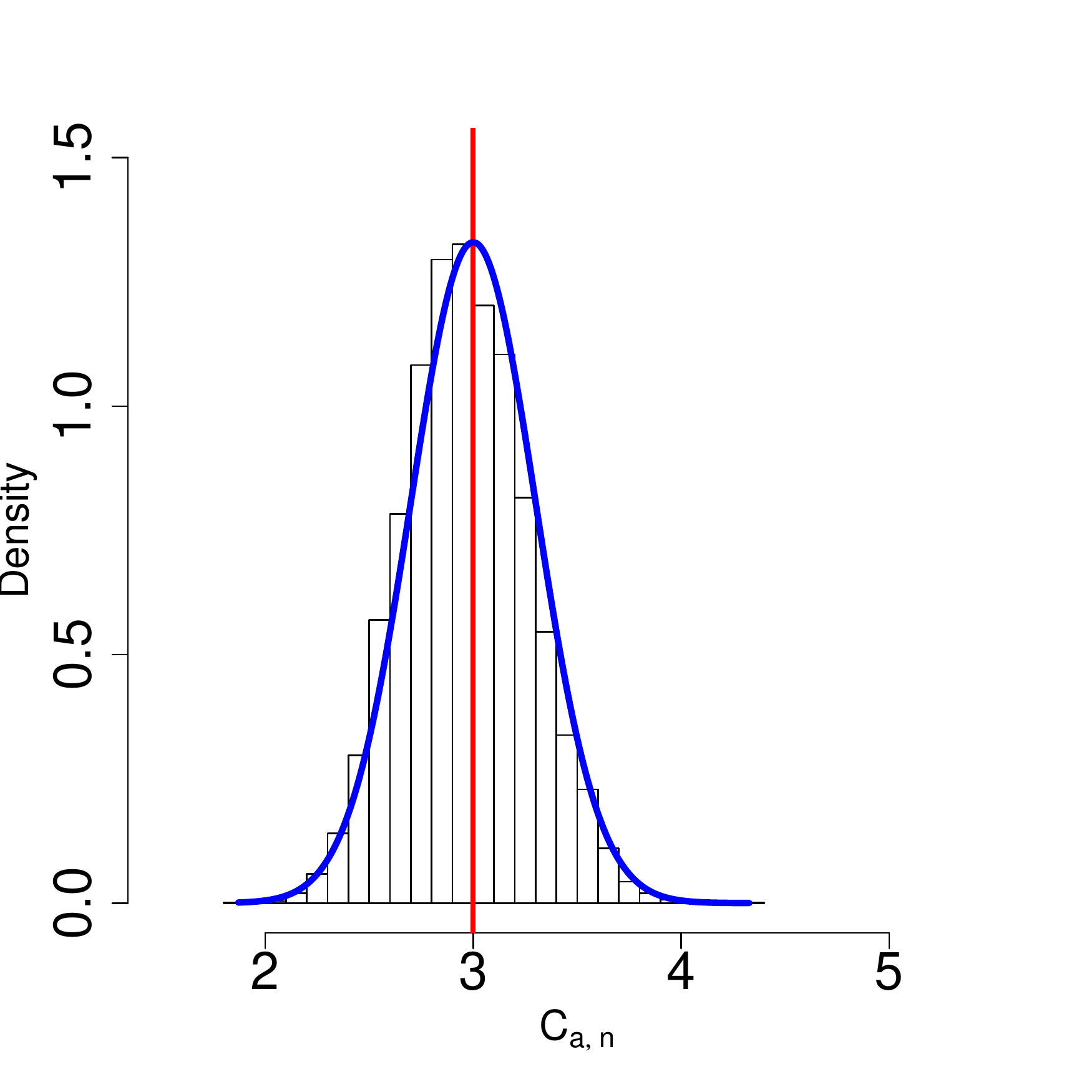}
\\
\includegraphics[width=5cm]{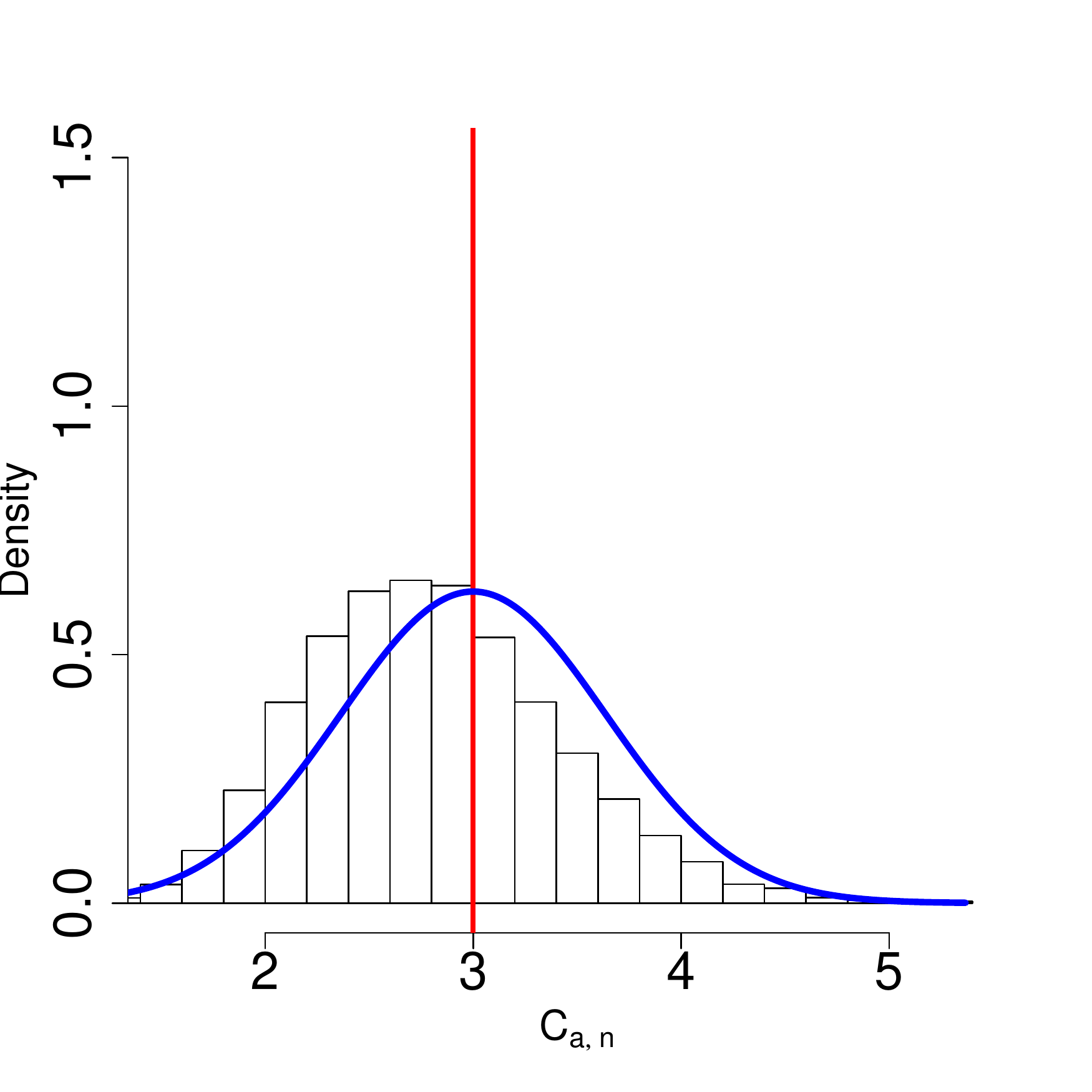}
&
\includegraphics[width=5cm]{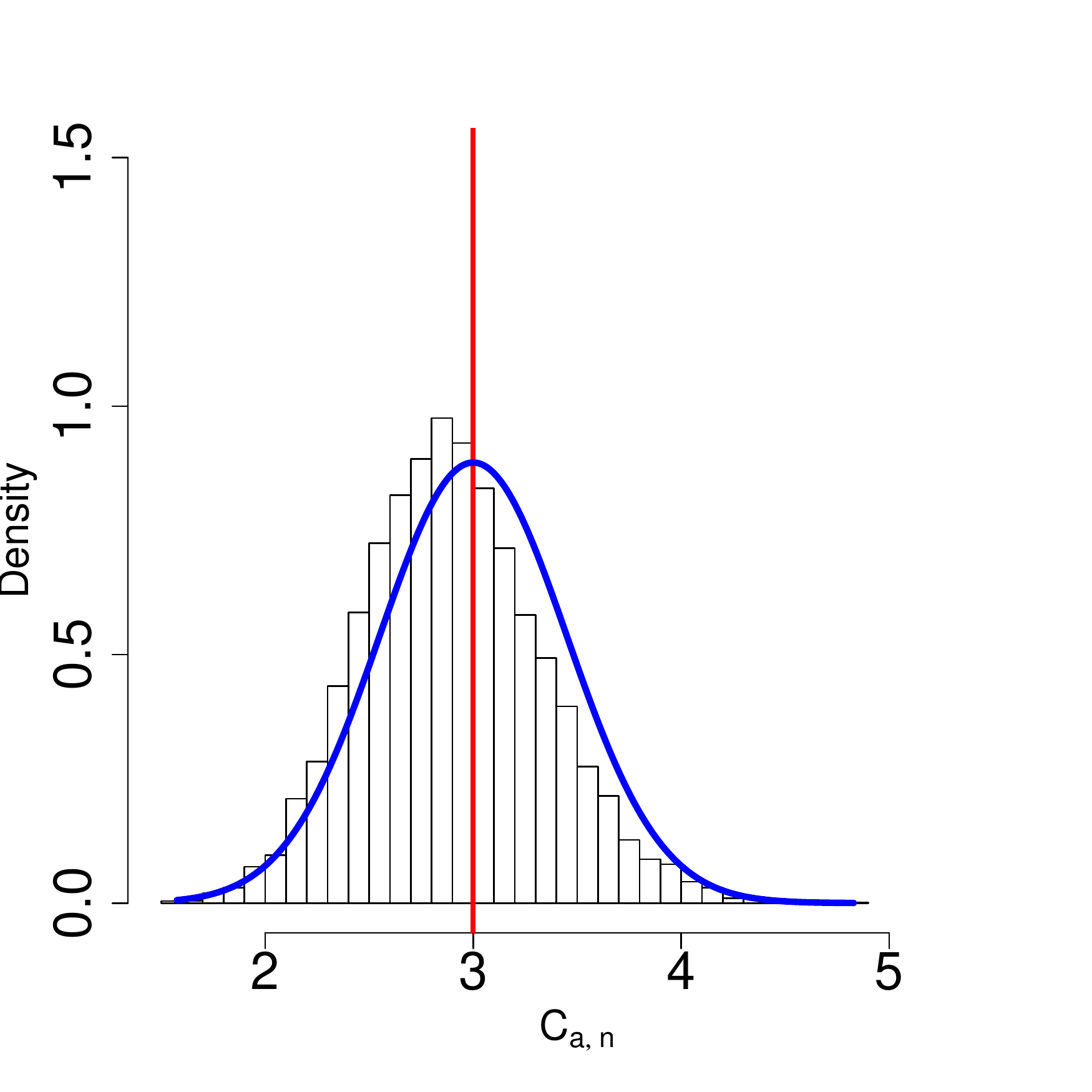}
&
\includegraphics[width=5cm]{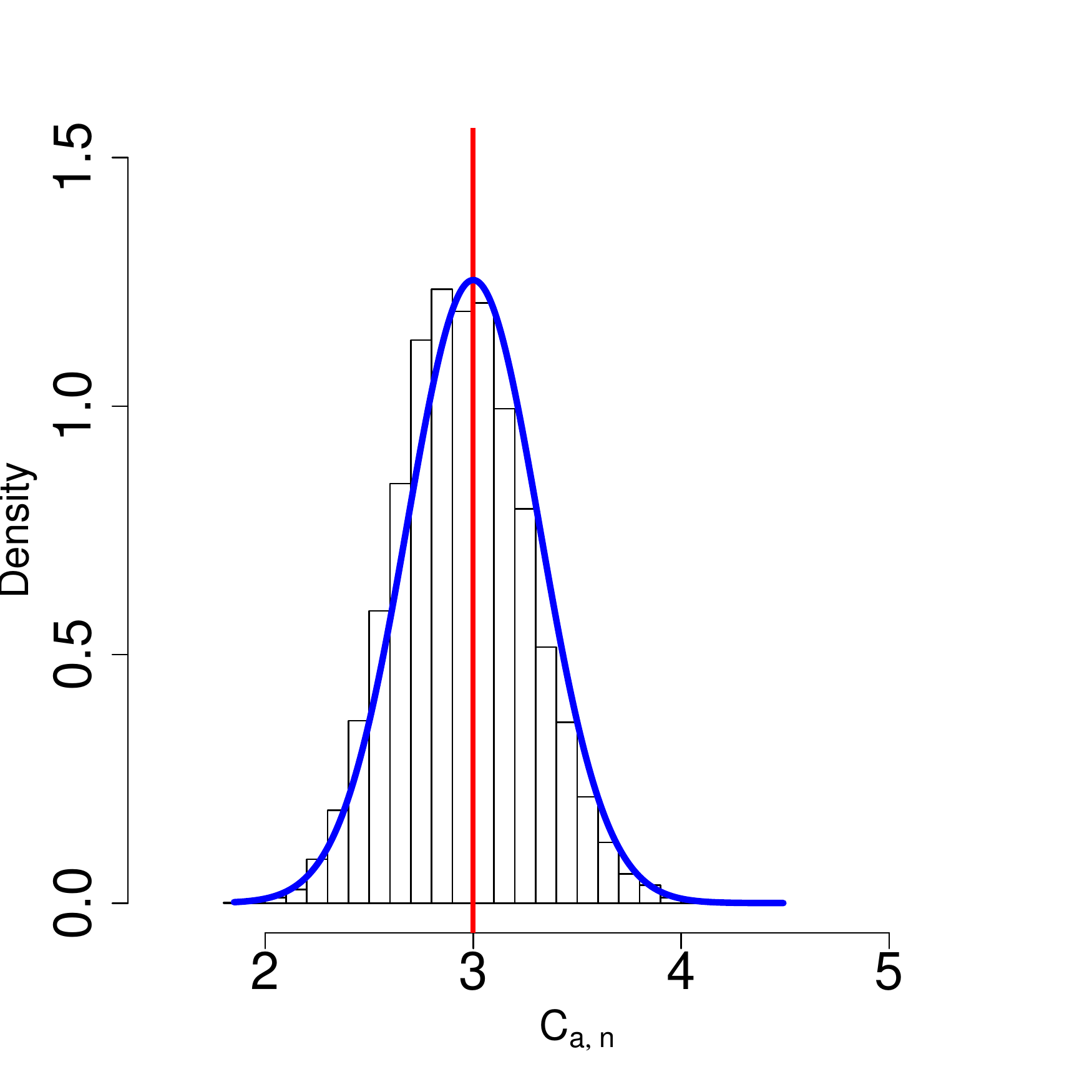}
\\
\includegraphics[width=5cm]{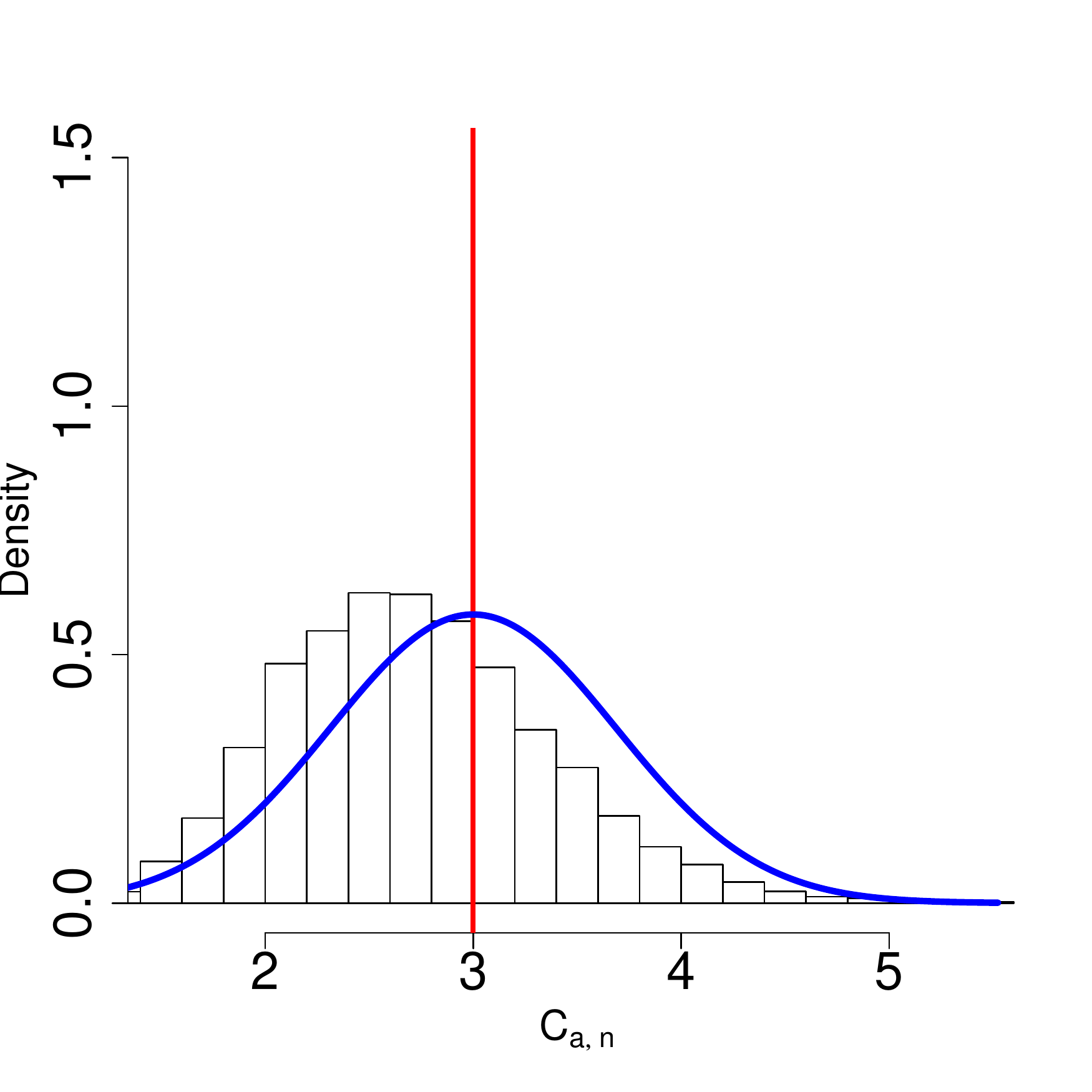}
&
\includegraphics[width=5cm]{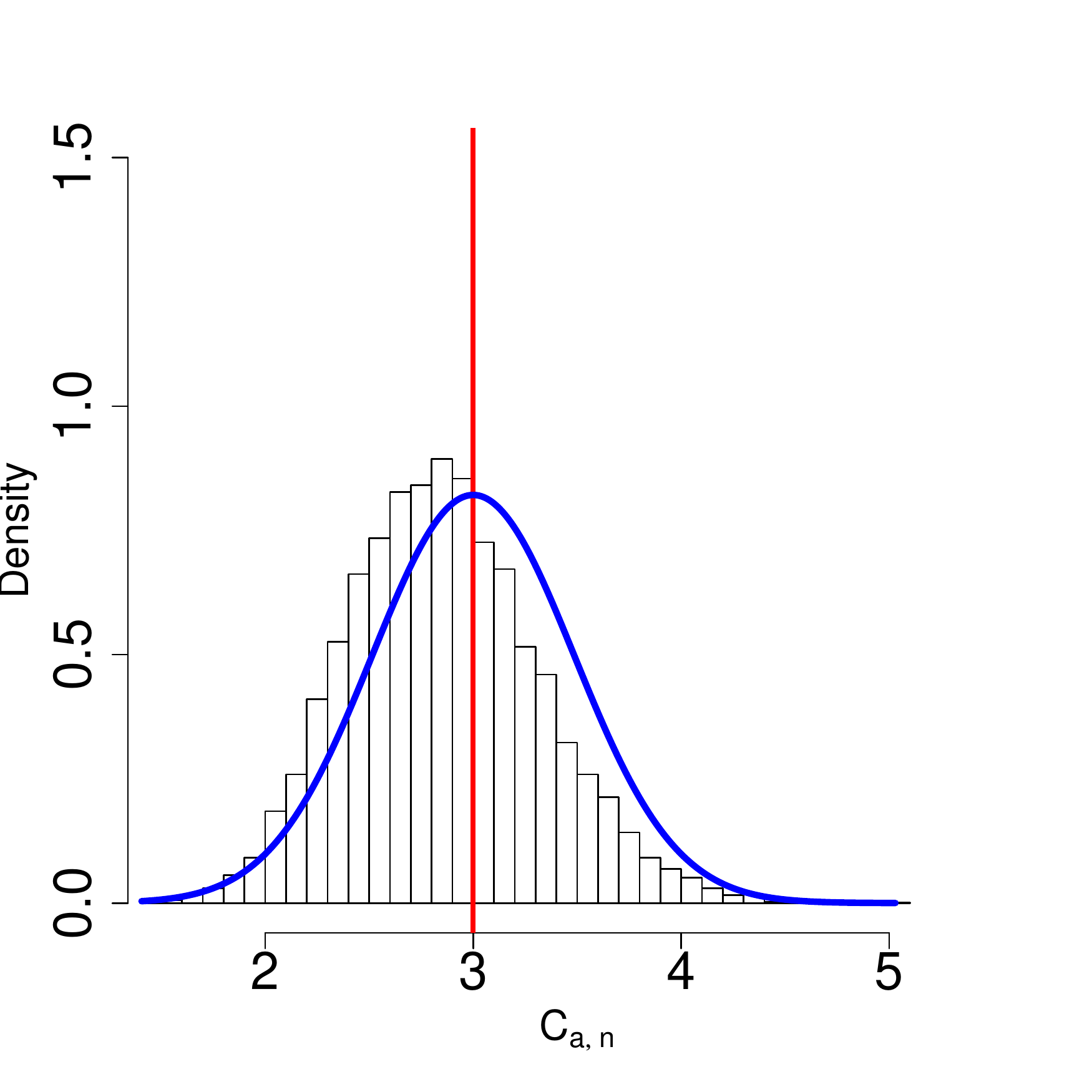}
&
\includegraphics[width=5cm]{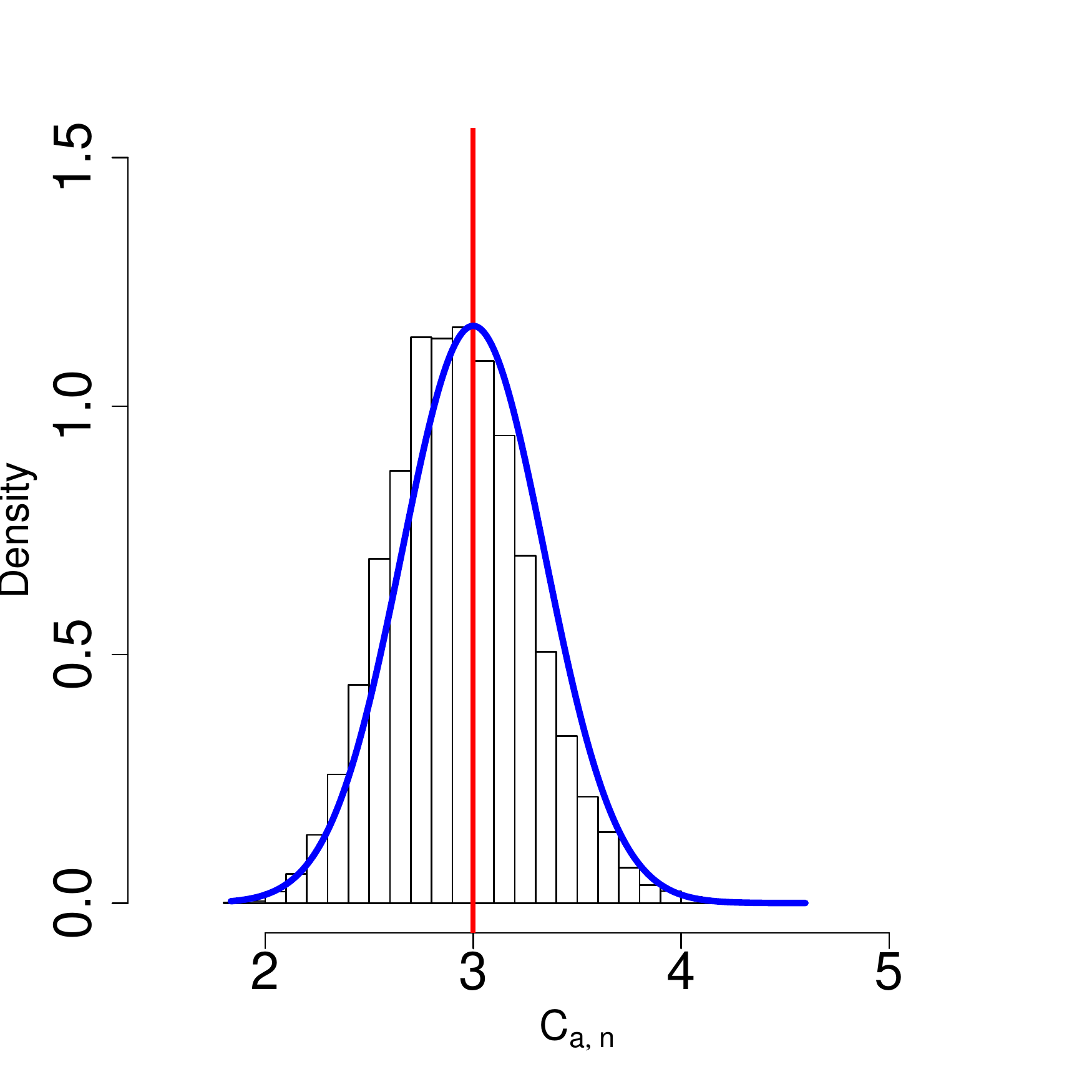}
\\
\end{tabular}

\caption{Comparison of the finite sample distribution of $C_{a,n}$ (histograms) with the asymptotic Gaussian distribution provided by Proposition \ref{prop:Van_Dqqe} and Theorem \ref{th:CLT_Van} (probability density function in blue line). The vertical red line denotes the true value of $C=3$. From left to right, $n=50,100,200$. From top to bottom, $D=0,1,2$. }
\label{fig:histograms}
\end{figure}

\subsection{Analysis of the asymptotic distributions}\label{ssec:simu_aggregation}

Now we consider the normalized asymptotic variance of $C_{a,n}$ obtained from \eqref{eq:var_van_Dqqe}  in Proposition \ref{prop:Van_Dqqe}. We consider the infill situation ($\delta_n = 1/n$, $\alpha=1$) and we let 

\begin{equation} \label{eq:tildeV}
\widetilde{v}_{a,s}
=
\frac{
 2 
\sum_{i\in \Z}  R^2(i,1,2D,\abs{\cdot{}}^{s},a^{2*}) 
}{
  R^2(0,1,2D,\abs{\cdot}^s,a^{2*}
},
\end{equation}
so that $(n^{1/2}/C)( C_{a,n} - C)$ converges to a $\mathcal{N}(0,\widetilde{v}_{a,s})$ distribution as $n \to \infty$, where $\widetilde{v}_{a,s}$ already defined in Section \ref{ssec:aggregation} does not depend on $C$ (nor on $n$). 

 \medskip

First, we consider the case $D=0$ and we plot $\widetilde{v}_{a,s}$ as a function of $s$ for various sequences $a$ in Figure \ref{fig:Dzero:var}.  The considered sequences are the following:
\begin{itemize}
\item the elementary sequence of order  $1$: $a ^{(1)}$ given by (-1,1);
\item the elementary sequence of order $2$: $a ^{(2)}$ given by (1,-2,1);
\item the elementary sequence of order $3$: $a ^{(3)}$ given by (-1, 3, -3, 1);
\item the elementary sequence of order $4$,  $a ^{(4)}$ given by (1,-4, 6,-4,1);
\item a sequence of order 1 and with length 3:  $a ^{(5)}$ given by (-1,-2,3);
\item a Daubechies wavelet sequence  \cite{daubechies1988orthonormal}  with $M=2$ as in \cite{IL97}:  $a ^{(6)}$ given by
 (-0.1830127,-0.3169873,1.1830127,-0.6830127);
\item a second Daubechies wavelet sequence with $M=3$:  $a ^{(7)}$ given by (0.0498175,0.12083221,-0.19093442,-0.650365,1.14111692,-0.47046721).
\end{itemize}\medskip

From Figure \ref{fig:Dzero:var}, we can draw several conclusions. First, the results of Section \ref{section:cramer:rao} suggest that $2$ is a plausible lower bound for $\widetilde{v}_{a,s}$. We shall call the value $2$ the Cramér-Rao lower bound. Indeed, we observe numerically that $\widetilde{v}_{a,s} \geq 2$ for all the $s$ and $a$ considered here.
Then we observe that, for any value of $s$, there is one of the $\widetilde{v}_{a,s}$ which is close to $2$ (below $2.5$). This suggests that quadratic variations can be approximately as efficient as maximum likelihood, for appropriate choices of the sequence $a$. We observe that, for $s=1$, the elementary sequence of order $1$ ($a_0=-1$, $a_1=1$) satisfies $\widetilde{v}_{a,s}  = 2$. This is natural since for $s=1$, this quadratic $a$-variations estimator coincides with the maximum likelihood estimator, when the observations stem from the standard Brownian motion. Except from this case $s=1$, we could not find other quadratic $a$-variations estimators reaching exactly the Cramér-Rao lower bound $2$ for other values of $s$.  

\medskip

Second, we  observe that the normalized asymptotic variance $\widetilde{v}_{a,s}$ blows up for the two sequences $a$ satisfying $M = 1$ when $s$ reaches $1.5$. This comes  from Remark \ref{rem:cas_pourris}: the variance of the quadratic $a$-variations estimators with $M=1$ is of order larger than $1/n$ when $s \geq 1.5$.
Consequently, we plot $\widetilde{v}_{a,s}$ for $0.1 \leq s \leq 1.4$ for these two sequences. For the other sequences satisfying $M\geq 2$, we plot $\widetilde{v}_{a,s}$ for $0.1 \leq s \leq 1.9$. 

\medskip

 Third, it  is difficult to extract clear conclusions  about the choice of the sequence: for $s$ smaller than, say, $1.2$ the two sequences with order $M = 1$ have the smallest asymptotic variance. Similarly, the elementary sequence of order $2$ has a smaller normalized variance than that of order $3$ for all values of $s$. Also, the Daubechies sequence of order $2$ has a smaller normalized variance than that of order $3$ for all values of $s$. Hence, a conclusion of the study in Figure \ref{fig:Dzero:var} is the following. When there is a sequence of a certain order for which the corresponding estimator reaches the rate $1/n$ for the variance, there is usually no benefit in using a sequence of larger order. Finally, the Daubechies sequences appear to yield smaller asymptotic variances than the elementary sequences (the orders being equal). The sequence of order $1$ given by $(a_0,a_1,a_2) = (-1,-2,3)$ can yield a smaller or larger asymptotic variance than the elementary sequence of order $1$, depending on the value of $s$. For two sequences of the same order $M$, it seems nevertheless challenging to explain why one of the two provides a smaller asymptotic variance.   \bigskip

\begin{figure}
\centering
\includegraphics[width=10cm]{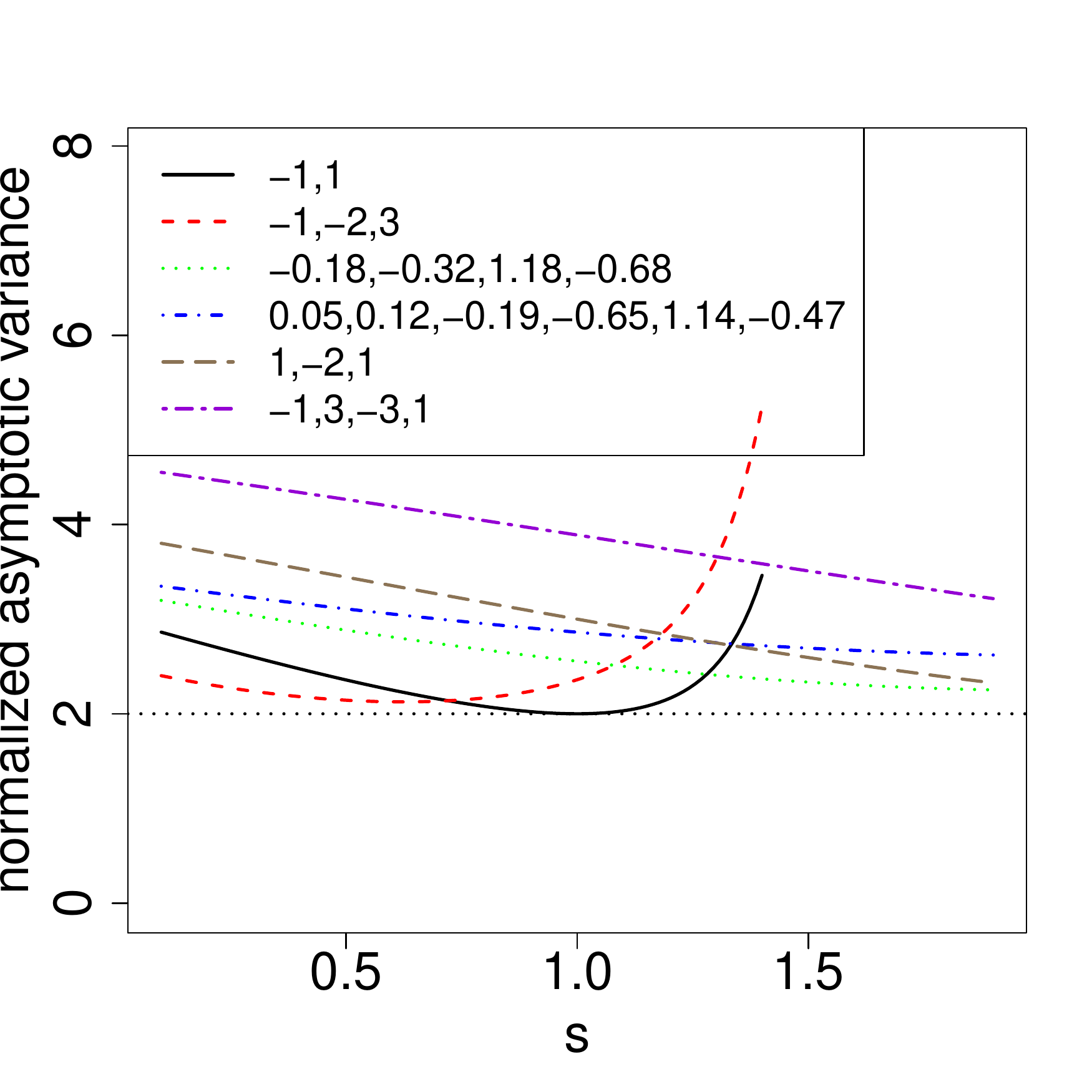}
\caption{Case $D=0$. Plot of the normalized asymptotic variance
$\widetilde{v}_{a,s}$ of the quadratic $a$-variations estimator, as a function of $s$, for various sequences $a$. The legend shows the values $a_0,...,a_{l}$ of these sequences (rounded to two digits).
From top to bottom in the legend, the
 sequences are the elementary sequence of order $1$, the sequence $(-1,-2,3)$ which has order $1$, the d
Daubechies sequences of order $2$ and $3$ and the elementary sequences of orders $2$ and $3$. The horizontal line corresponds to the Cramér-Rao lower bound $2$. 
}
\label{fig:Dzero:var}
\end{figure}

Now, we consider aggregated estimators, as presented in Section \ref{ssec:aggregation}. A clear motivation for considering aggregation is that, in Figure \ref{fig:Dzero:var}, the smallest asymptotic variance $\widetilde{v}_{a,s}$ corresponds to different sequences $a$, depending on the values of $s$.

In Figure \ref{fig:Dzero:aggreg} left, we consider the case $D=0$ and  we use four sequences:  $a^{(1)}$,  $a^{(5)}$ $a^{(2)}$ and $a^{(6)}$. We plot their  corresponding asymptotic variances $\widetilde{v}_{a^{(i)},s}$ as a function of $s$, for $0.1 \leq s \leq 1.4$ as well as the variance of their aggregation. It is then clear that aggregation drastically improves each of the four original estimators. The asymptotic variance of the aggregated estimator is very close to the Cramér-Rao lower bound $2$ for all the values of $s$. 
In Figure \ref{fig:Dzero:aggreg} right,  we perform the same analysis but with  sequences of order larger than 1. The  four considered sequences are now $a^{(6)}$,  $a^{(2)}$ $a^{(3)}$ and $a^{(4)}$. The value of 
$s$ varies from  $0.1$ to $ 1.9$ 
Again, the aggregation  is clearly  the best. 

\medskip

Eventually, Figures \ref{fig:Dun:var} and \ref{fig:Dun:aggreg} explore the case $D=1$. Conclusions are similar.

 \begin{figure}
\centering
\begin{tabular}{cc}
\includegraphics[width=7cm]{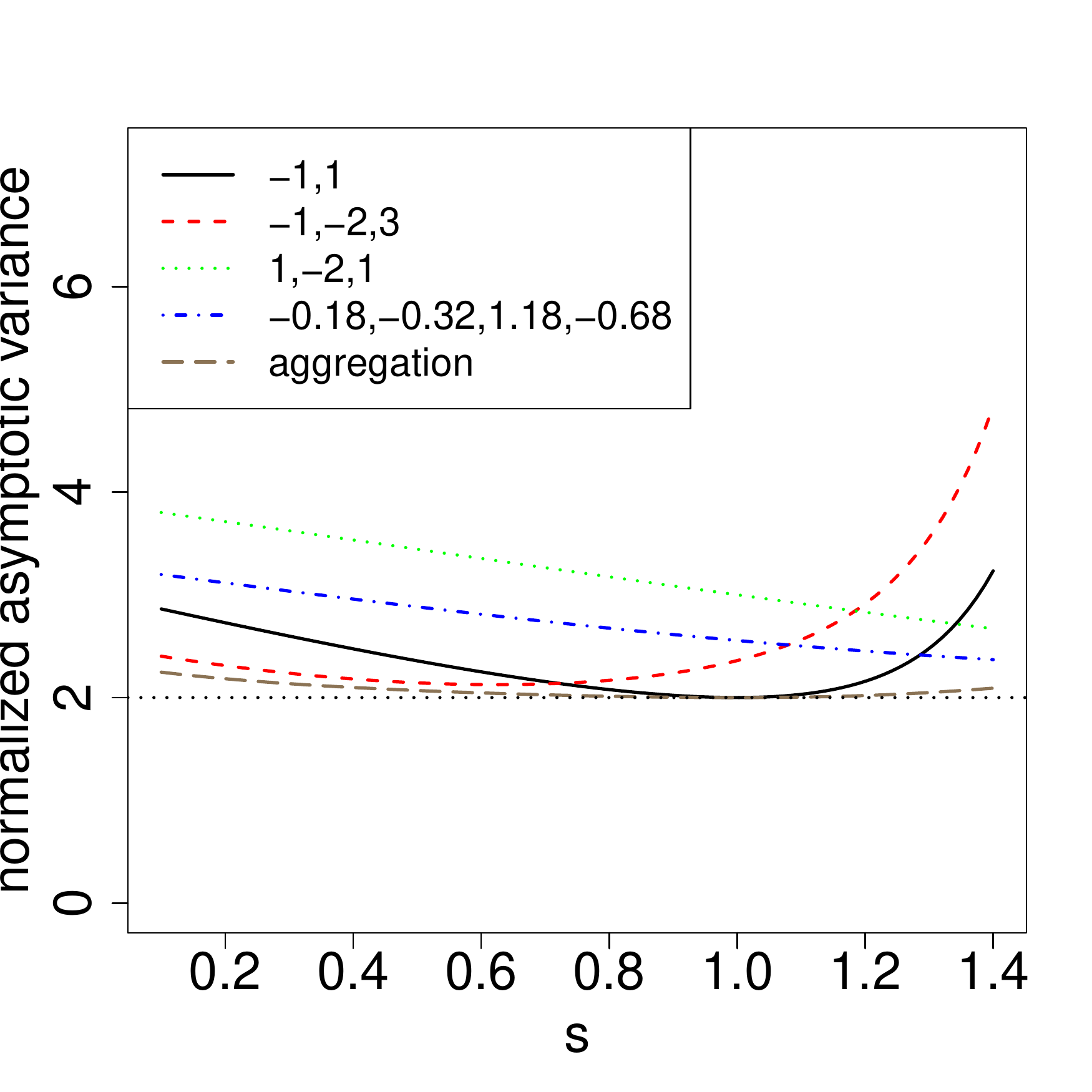}
&
\includegraphics[width=7cm]{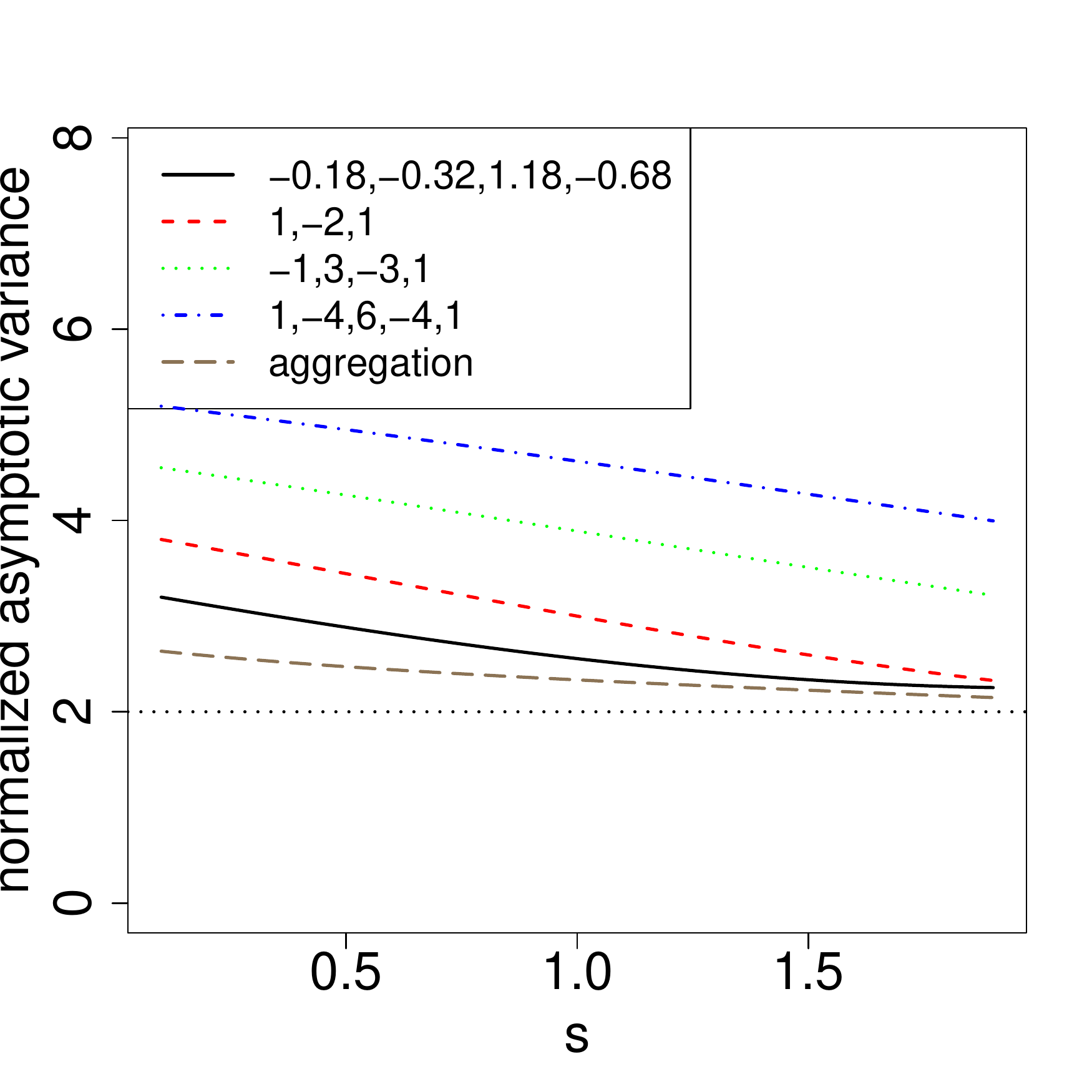}
\end{tabular}

\caption{Case $D=0$. Plot of the normalized asymptotic variance $\widetilde{v}_{a,s}$ of the quadratic $a$-variations estimator, as a function of $s$, for various sequences $a$ and for their aggregation.
On the left, including the order one elementary sequence, on the right without. 
The horizontal line corresponds to the Cramér-Rao lower bound $2$. }
\label{fig:Dzero:aggreg}
\end{figure}

%
%
\begin{figure}
\centering
\includegraphics[width=10cm]{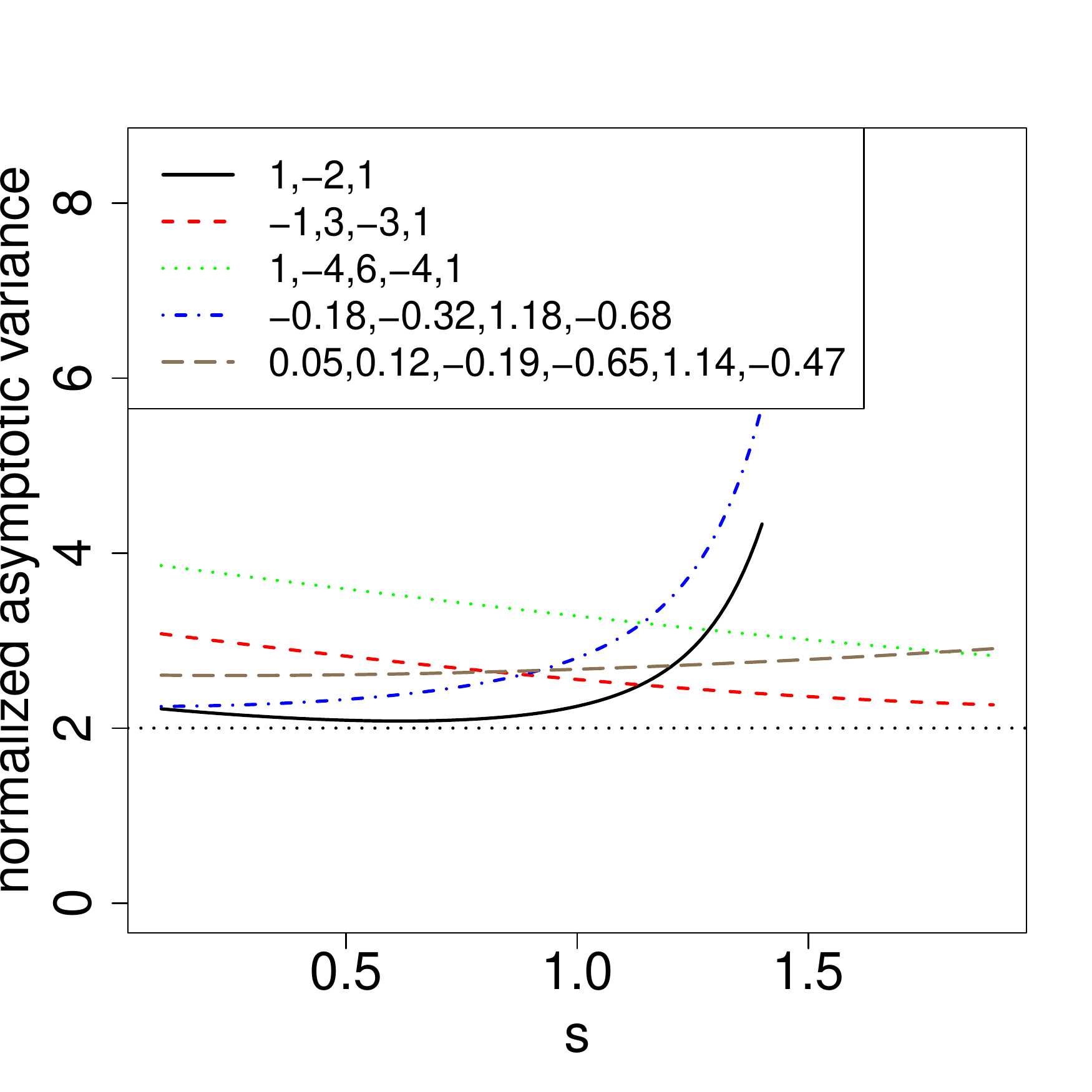}
\caption{Same setting as in Figure \ref{fig:Dzero:var} but for $D=1$. From top to bottom in the legend, the sequences are the elementary sequences of order $2$, $3$  and $4$ and the Daubechies sequences of order $2$ and $3$.}
\label{fig:Dun:var}
\end{figure}

%

\begin{figure}
\centering
\begin{tabular}{cc}
\includegraphics[width=7cm]{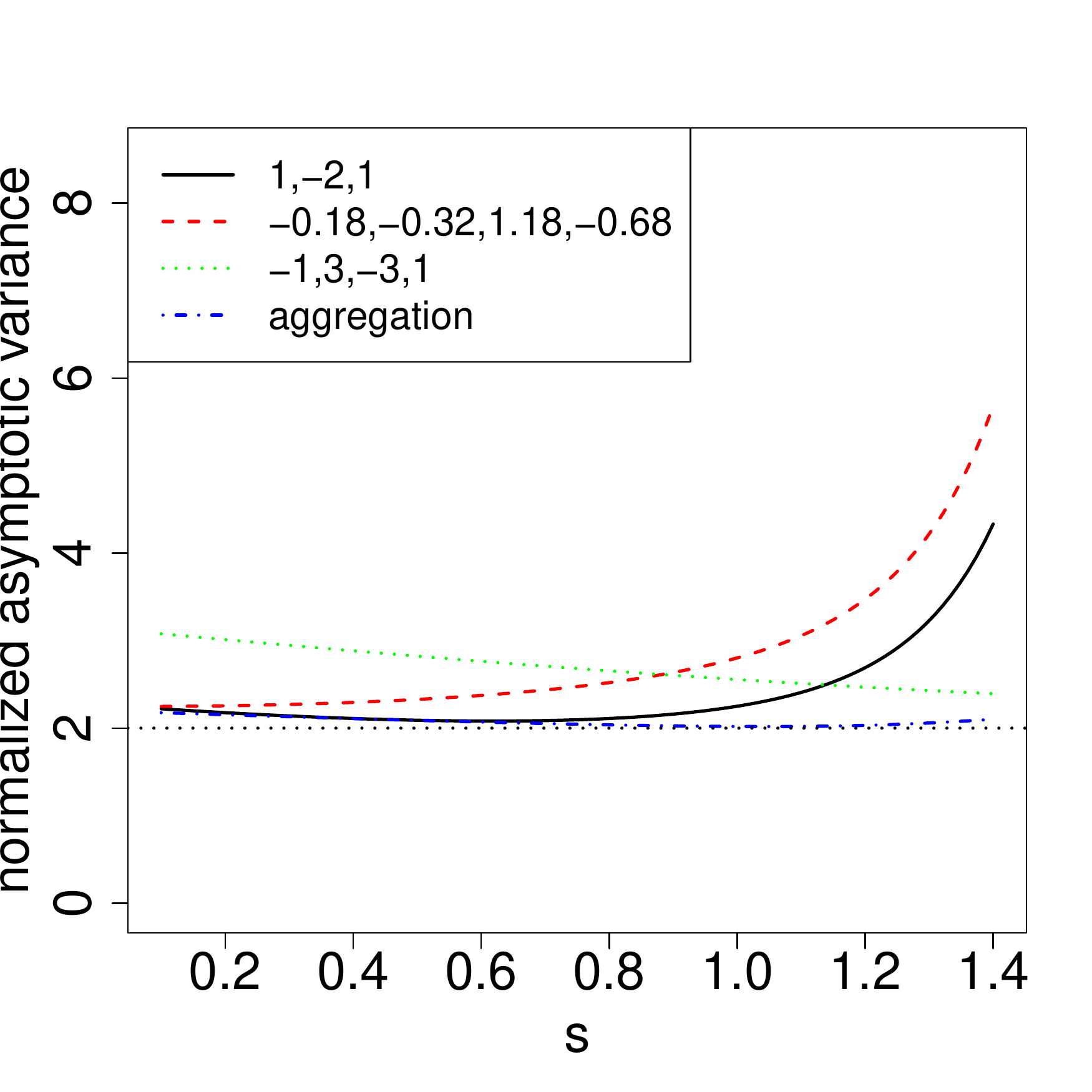}
&
\includegraphics[width=7cm]{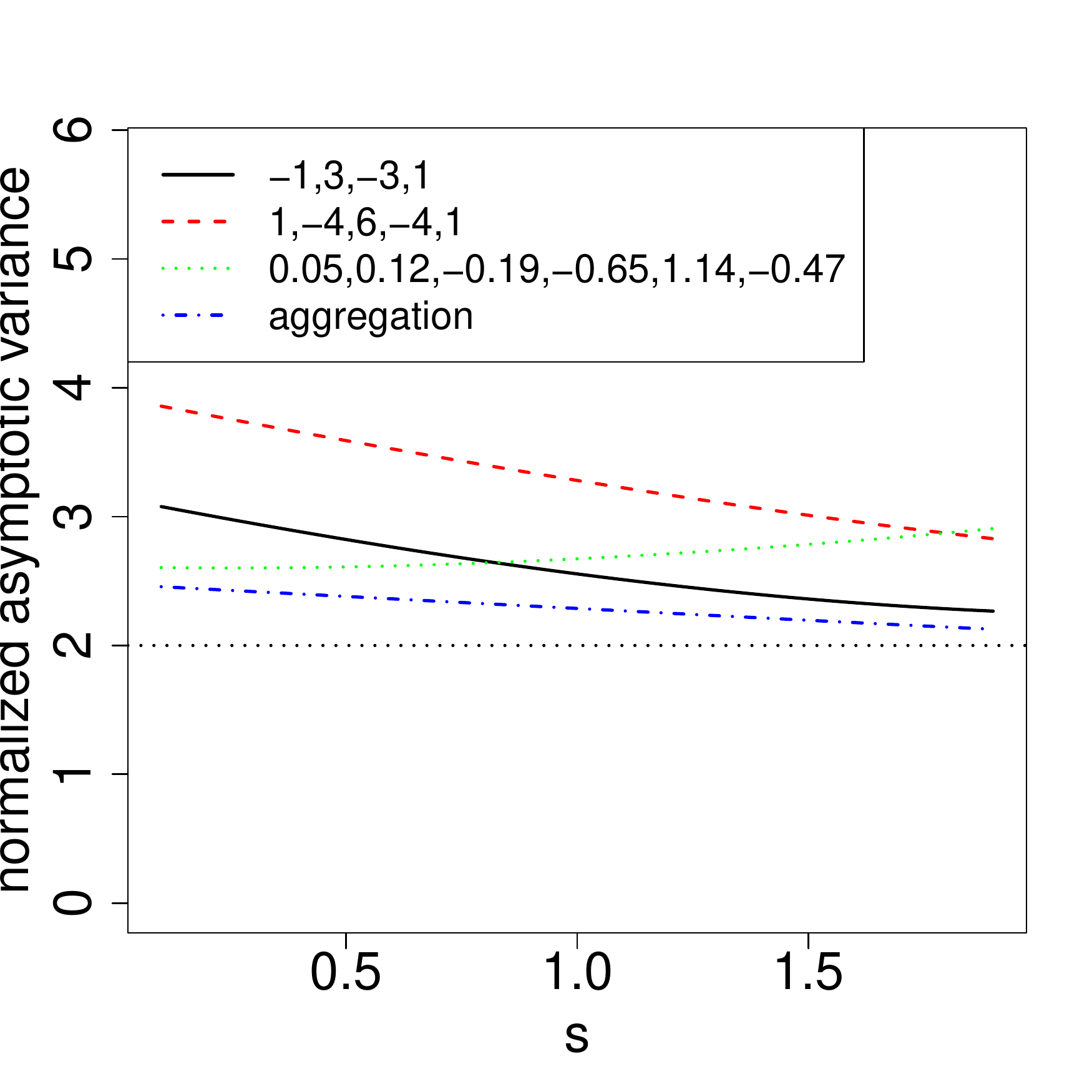}
\end{tabular}

\caption{Same setting as in Figure \ref{fig:Dzero:aggreg} but for $D=1$. On the left, from top to bottom in the legend, the sequences are the elementary sequence of order $2$, the Daubechies sequence of order $2$ and the  elementary sequence of order $3$.
On the right, from top to bottom in the legend, the sequences are the elementary sequences of orders $3$ and $4$ and the Daubechies sequence of order $3$.
}
\label{fig:Dun:aggreg}
\end{figure}

\subsection{Real data examples} \label{ssec:real:data}

In this section, we consider real data of spatially distributed processes in dimension two. In this setting, we extend the estimation procedure based on the quadratic $a$-variations that is then compared to the MLE procedure.

\subsubsection{A moderate size data set} \label{subsubsection:small:data:set}

We compare two methods of estimation of the autocovariance function of a separable Gaussian model on a real data set of atomic force spectroscopy\footnote{Personal communication from C. Gales and J. M. Senard.}.
The data consist of observations taken on a grid of step $1/15$ on $[0,1]^2$, so they consist of $256$ points of the form 
\[
X( i/15 , j/15 )
~ ~
i=0,\ldots,15,
~
j=0,\ldots,15.
\]

The first method is maximum likelihood estimation in a Kriging model, obtained from the function \verb#km# or the \verb#R# toolbox \verb#DiceKriging# \cite{roustant12dice}. For this method, the mean and autocovariance functions are assumed to be $\mathbb{E}(X(i/15,j/15)) = \mu$ and
\begin{equation} \label{eq:cov:km}
\mathrm{Cov} ( X(i/15,j/15) ,  X(i'/15,j'/15) )
=
\sigma^2
e^{- \theta_1 | i-i' |/15}
e^{- \theta_2 | j-j' |/15}.
\end{equation}
The parameters $\mu,\sigma^2,\theta_1,\theta_2$ are estimated by maximum likelihood.

The second method assumes the same autocovariance model \eqref{eq:cov:km} and consists in the following steps.
\begin{itemize}
\item[(1)] Estimate $\sigma^2$ by the sum of square
\[
\hat{\sigma}^2
=
\frac{1}{256}
\sum_{i,j=0}^{15}
( X(i/15,j/15)  - m )^2
\]
with $m = \sum_{i,j=0}^{15}
 X(i/15,j/15) $.
 \item[(2)] For each column $j$ of $[X(i/15,j/15)]_{i,j=0,\ldots,15}$, the vector of $16$ observations obey our model with $s=1$ and $C_1 = \sigma^2 \theta_1$. Hence, we can estimate $C_1$ by $\hat{C}_{1,j}$ with the estimator \eqref{eq:Can}, with the elementary sequence of order $1$. We thus obtain an estimate $\hat{C}_{1}$ by averaging the $\hat{C}_{1,j}$ for $j=0,\ldots,15$. 
  \item[(3)] We perform the same analysis row by row to obtain an estimate $\hat{C_2}$.
  \item[(4)] For $i = 1,2$, $\theta_i$ is estimated by $\hat{\theta}_i = \hat{C}_i / \hat{\sigma}^2$.
\end{itemize}

The first method, based on maximum likelihood, provides infinite values for $\theta_1$ and $\theta_2$, so that it considers the $256$ observed values as completely spatially independent. On the other hand, the second method provides the values $\hat{\theta}_1 = 14.72$ and $\hat{\theta}_2 = 15.73$. This corresponds to a correlation of approximately $1/e \approx 0.36$ between direct neighbors on the grid. Hence, the second method, based on our suggested quadratic variation estimator, is able to detect a weak correlation (that can be checked graphically), but not the maximum likelihood estimator.

\subsubsection{A large size data set} \label{subsubsection:large:data:set}

The second data set consists in a two-dimensional field of deformation amplitude, corresponding to the registration of two real images. The deformation field is obtained from the software presented in \cite{risser:11:adni}.
Figure \ref{fig:images:laurent} displays the two images to be registered and the deformation field. 

\begin{figure}

\centering
\begin{tabular}{ccc}
\includegraphics[width=5cm,height=5cm]{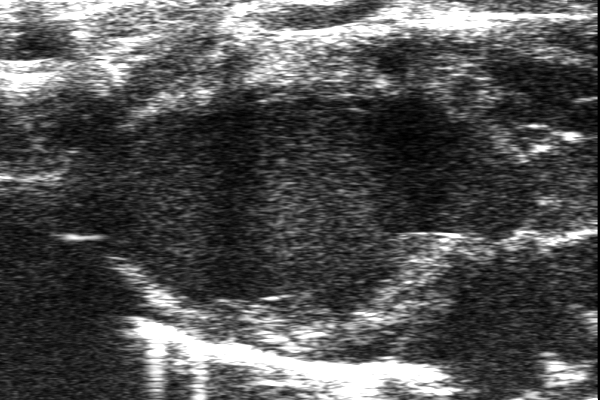}
&
\includegraphics[width=5cm,height=5cm]{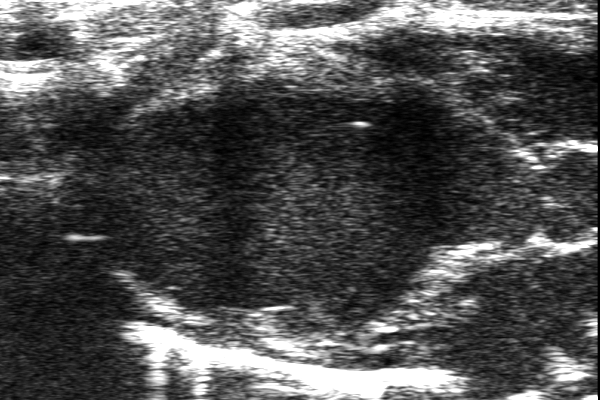}
&
\includegraphics[width=5cm,height=5cm]{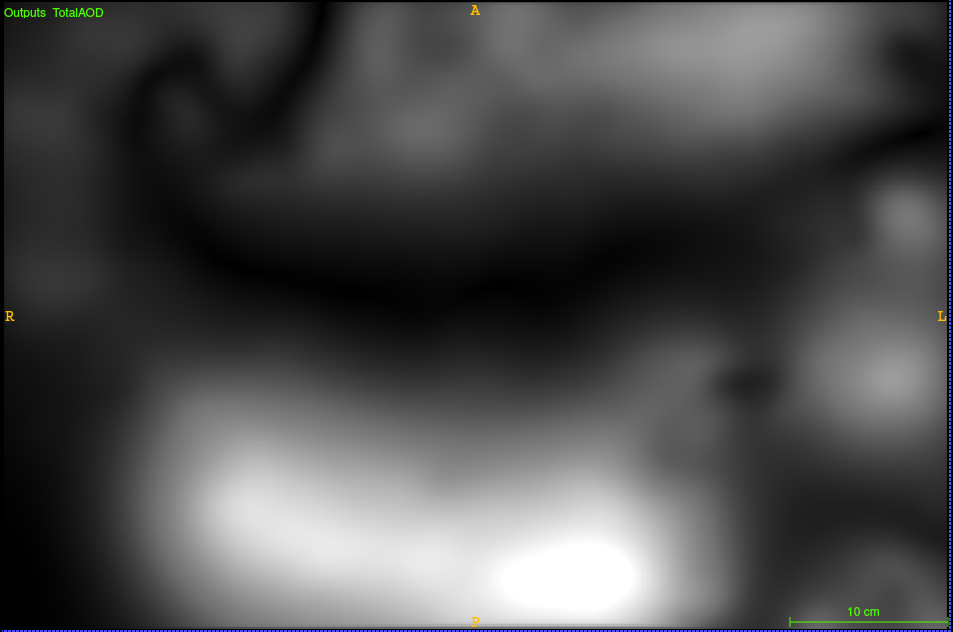}
\end{tabular}
\caption{For the data set of Section \ref{subsubsection:large:data:set}: the two images to be registrated (left and middle) and the field of deformation amplitude (right). On the right, light colors indicate large deformation amplitudes and dark colors indicate small deformation amplitudes.}
 \label{fig:images:laurent}
\end{figure}

After a subsampling of the field of deformation amplitude, the data consist of observations taken on a rectangular grid of steps $1/56$ and $1/59$ on $[0,1]^2$, so they consist of $3420$ points of the form 
\[
X( i/56 , j/59 )
~ ~
i=0,\ldots,56,
~
j=0,\ldots,59.
\]

With these data, we consider the same autocovariance model as in Section \ref{subsubsection:small:data:set}. We estimate the parameters $\theta_1$ and $\theta_2$ from the same two methods as in Section \ref{subsubsection:small:data:set}. The first method provides $\hat{\theta}_1 = 0.8770$ and $\hat{\theta}_2 = 0.6547$ and takes about 22 minutes on a personal computer. The second method provides $\hat{\theta}_1 = 0.607$ and $\hat{\theta}_2 = 0.107$ and takes about $0.05$ seconds on a personal computer. Hence, our suggested quadratic variation estimator provides a very significant computational benefit.

Both estimators conclude that the spatial correlation is more important along the $x$-axis that along the $y$-axis, which is graphically confirmed in Figure \ref{fig:images:laurent}. As in Section \ref{subsubsection:small:data:set}, the maximum likelihood estimator provides less correlation than the quadratic variation estimator.

Finally, if the field of deformation is considered with no preliminary subsampling, its size is $400 \times 600$. In this case, the MLE can not be directly implemented while the quadratic variation estimator can be.

\section{Conclusion}
\label{section:conclusion}

We have provided an in-depth analysis of the estimation of the scale parameter of a one-dimensional Gaussian process by quadratic variations. Indeed, the knowledge of this scale parameter is essential when studying a Gaussian process, as it enables to quantify its dependence structure, or to test  independence. 

We have addressed a semi-parametric setting, where no parametric family of variograms needs to be assumed to contain the unknown variogram. We have suggested an estimator, based on previous references, which numerical implementation is straightforward. Our theoretical analysis follows the principles of previous references, but is significantly simpler and holds under mild and simple to check technical assumptions. Based on this theoretical analysis, we have been able to tackle more advanced statistical topics, such as the aggregation of estimators based on different sequences, in the aim of improving the statistical efficiency.

Our analysis paves the way for further research topics. For instance, it would be interesting to estimate the variances and covariances of a set of quadratic variation estimators, in order to estimate the optimal aggregation of them.


\appendix

\section{Proofs}\label{sec:app}

\subsection{Proof of the results of Section \ref{ssec:res_a_var}}

\begin{proof}[Proof of Proposition \ref{prop:ifbm}] By the stochastic Fubini theorem,
 \begin{align*}
 B_s^{(-m)}(u_1) &=  \int_0 ^{u_1} du_2 \dots  \int_0 ^{u_m} du_{m+1}   \int_{-\infty} ^{u_{m+1}} dW(t) f_s( t,u_{m+1})\\
& =  
 \int_{-\infty}^{u_1}    dW(t)  \int_t ^{u_1} du_2 \dots   \int_t^{u_m}  du_{m+1} f_s( t,u_{m+1}) \\
& =:   \int_0 ^{u_1}   g_{m,s}  (u_1,t)dW(t) .
 \end{align*}
The positiveness  of $f_s(t,u)$ for  $u>0$ implies that  of $g_{m,s}(t,u)$. As a consequence,  for  $0<t_1<\dots t_k$, $B_s^{(-m)}(t_k) $ includes  a non-zero component:
\[
      \int_{t_{k-1} }^{t_k}  g_{m,s}  (u,t)dW(t),
\]
which is independent  of $(B_s^{(-m)}(t_1),\dots, B_s^{(-m)}(t_{k-1}))$ implying that $B_s^{(-m)}(t_k) $  is not collinear to this set of variables.  By induction,  this implies in turn that $B_s^{(-m)}(t_1),\dots ,B_s^{(-m)}(t_{k}) $ are not collinear.

\end{proof}

\begin{proof}[Proof of Lemma \ref{l:jma}] For $m=0$, we have 
\[
 \Var  \big( B_s^{(-0)}(u) -  B_s^{(-0)}(v) \big)
 = 2| u-v |^s
\]
so that the lemma holds with the convention $(s+1)\dots (s+0) = 1$. Thus we prove it by induction on $m$ and assume that it holds for $m \in \mathbb{N}$. We have, with  $K^{(-r)}(u,v)
 = 
 \mathbb{E}  \big[ B_s^{(-r)}(u)  B_s^{(-r)}(v) \big]$, for $r \in \mathbb{N}$,
 \begin{align*}
 K^{(-m)}(u,v)
 =  & \frac{1}{2}
 \left(
 \Var  \big( B_s^{(-m)}(u) -  B_s^{(-m)}(0) \big)
 +
 \Var  \big( B_s^{(-m)}(v) -  B_s^{(-m)}(0) \big)
 -
 \Var  \big( B_s^{(-m)}(u) -  B_s^{(-m)}(v) \big)
 \right)
 \\
 = &
  \psi(u) + \psi(v)
  - \frac{1}{2} \sum_{i=1}^{N_m}
 P^{m,i}(v) h_{m,i}(u)
    - \frac{1}{2} \sum_{i=1}^{N_m}
 P^{m,i}(u) h_{m,i}(v)  
 - \frac{1}{2} (-1)^m \frac{ 2|u-v|^{s+2m}}{
 (s+1)\dots (s+2m)},
 \end{align*}
 where $\psi$ is some function. Since we have $K^{(-(m+1))}(u,v) =  \int_{0}^u\int_{0}^v K^{(-m)}(x,y) dx dy$,
  \begin{align} \label{eq:Kmmp1:fbM}
 K^{(-(m+1))}(u,v)
 & =
  \sum_{i=1}^{\widetilde{N}_{m+1}}
 \widetilde{P}^{m+1,i}(v) \widetilde{h}_{m+1,i}(u)
  + \sum_{i=1}^{\widetilde{N}_{m+1}}
 \widetilde{P}^{m+1,i}(u) \widetilde{h}_{m+1,i}(v) \notag
 \\
 & 
 + (-1)^{m+1} \frac{1}{ (s+1)\dots (s+2m)} \int_{0}^v \left( \int_{0}^u |x-y|^{s+2m} dx \right) dy,
 \end{align}
 where $\widetilde{N}_{m+1} \in \mathbb{N}$, where for $i=1,...,\widetilde{N}_{m+1}$, $\widetilde{P}^{m+1,i}$ is a polynomial of degree less or equal to $m+1$ and $\widetilde{h}_{m+1,i}$ is some function. For $v \leq u$, we have
 \begin{align*}
 \int_{0}^v \left( \int_{0}^u |y-x|^{s+2m} dx \right) dy
 = &
  \int_{0}^v \left( 
  \int_{0}^y (y-x)^{s+2m} dx 
+  
 \int_{y}^u (x-y)^{s+2m} dx 
  \right) dy
  \\
   = &
  \int_{0}^v \left( 
  \frac{y^{s+2m+1} }{2m+1}
+  
  \frac{(u-y)^{s+2m+1} }{2m+1}
  \right) dy
  \\
  = &  
  \frac{v^{s+2m+2} }{(2m+1)(2m+2)}
  -
  \frac{(u-v)^{s+2m+2} }{(2m+1)(2m+2)}
  +
    \frac{u^{s+2m+2} }{(2m+1)(2m+2)}.
 \end{align*}
 By symmetry, we obtain, for $u,v \in \mathbb{N}$,
 \begin{equation} \label{eq:Impun:fbM}
 \int_{0}^u \left( \int_{0}^v |x-y|^{s+2m} dx \right) dy
 =
 \frac{u^{s+2m+2} }{(2m+1)(2m+2)}
  +
    \frac{v^{s+2m+2} }{(2m+1)(2m+2)}
    -
  \frac{|u-v|^{s+2m+2} }{(2m+1)(2m+2)}.
 \end{equation}
 Hence, from the relation 
\[
 \Var  \big( B_s^{(-(m+1))}(u) -  B_s^{(-(m+1))}(v) \big) = K^{(-(m+1))}(v,v) + K^{(-(m+1))}(u,u) - 2 K^{(-(m+1))}(v,u),
\] 
 \eqref{eq:Kmmp1:fbM}, and  \eqref{eq:Impun:fbM}, we conclude the proof of the lemma.
\end{proof}

\begin{proof}[Proof of Proposition \ref{prop:sumneq0}]
Using Lemma \ref{l:jma} (with $m=D$) and the  vanishing moments of  $a$ of order less or equal  than $D$, we have 

\begin{align*}
\sum_{k,l}a_ka_l \abs{k-l}^{2D+s} &=  (Const) (-1)^D\sum_{k,l}a_ka_l  \Var(  B_s^{(-D)}(k) -  B_s^{(-D)}(l))  \\
 &=(Const) (-1)^{D+1} \Var \left( \sum_k  a_k B_s^{(-D)}(k) \right).
 \end{align*}
We conclude using the ND property of the IFBM stated in Proposition \ref{prop:ifbm}. 
\end{proof}

\subsection{Preliminary results}

\begin{lemma}\label{lem:mehler} Let  $Z=(X,Y)$ be a centered Gaussian vector of dimension 2  then 
\[
 \Cov\left(X^2,Y^2\right)=2 \Cov^2\left(X,Y\right).
\]
\end{lemma}

\begin{proof}[Proof of Lemma \ref{lem:mehler}]
This Lemma is a consequence of the so called Mehler formula \cite{AW09}. Its proof is immediate using the cumulant method.
\end{proof}

\begin{lemma}\label{lem:cov} Assume that $V$ satisfies $\left(\mathcal{H}_{0}\right)$, $\left(\mathcal{H}_{1}\right)$ and $\left(\mathcal{H}_{2}\right)$. One has,
when $M >D +s +1/4$,
	\[
\max_{i=1,\dots,n'}
	 \left(\sum_{i'= 1,\dots,n'}\abs{\Sigma_a(i,i')}\right)  = o\left(\Var(V_{a,n})^{1/2}\right).
	\]
\end{lemma}

\begin{proof}[Proof of Lemma \ref{lem:cov}]    Using the stationary increments of the process, one has
\begin{align} \label{e:z:1}
\max_{i=1,\dots,n'}
	 \left(\sum_{i'= 1,\dots,n'}\abs{\Sigma_a(i,i')}\right)  
\leq  2\sum_{i=0}^{n'-1} \abs{\Sigma_a(1,1+i)}.
\end{align}
Recall that 
\[
\Sigma_a(1,1+i) =\Cov \left(\Delta_{a,1}(X), \Delta_{a,1+i}(X) \right) 
= -\Delta_{a^{2*},i}(V) =
\delta_n^{2D}  R(i,\delta_n,2D,V^{(2D)},a^{2*}).
\]
We have seen in  the proof of  Proposition \ref{prop:Van_Dqqe} (\eqref{e:a:0} and \eqref{e:a:1}) 
that  for $i$ sufficiently large 
\[
R(i,\delta_n,2D,V^{(2D)},a^{2*}) \leq (Const) \big(\delta_n^s i^{s-2(M-D)} + \delta_n^{d+\beta} i^\beta\big).
\]
 Thus the sum in \eqref{e:z:1} is bounded by
\[
 (Const) \delta_n^{2D+s} (n^{s-2(M-D) +1} + 1)  + (Const)  \delta_n^{2D+d+\beta}  (n^{1+\beta}+1).
\]
On  the other hand, we have proved also in the proof of Proposition \ref{prop:Van_Dqqe} that 
\[
\Var (V_{a,n}) ^{1/2} = (Const)  n^{1/2}   \delta_n^{2D+s}(1+o(1))
\]
giving the result. Thus, one has to check that
\[
\delta_n^{2D+s} n^{s-2(M-D) +1}, \quad \delta_n^{2D+s}, \quad \delta_n^{2D+d+\beta} n^{1+\beta}, \quad \textrm{and} \quad  \delta_n^{2D+d+\beta} 
\]
are $o(n^{1/2}\delta_n^{2D+s})$ which is true by the assumptions made. We skip the details.
\end{proof}

\subsection{Proof of the main results}

\begin{proof}[Proof of Proposition \ref{prop:Van_Dqqe}] {\bf 1)}
By definition of $V_{a,n}$ in \eqref{def:Van} and  identity \eqref{eq:prop3_aa},  we get
\begin{align}\label{e:zaza1}
\E[V_{a,n}] = &n' \E [\Delta_{a,i}(X)^2 ] = -n' \Delta_{a^{2*},0}(V) 
= - n'\sum_{j} a_j^{2*} V(j\delta_n).
\end{align}
Recall that $n' =n-L +1$ is the size of the vector $\mathbf{\Delta_a}(X)$. In all the proof, $j$ is assumed to vary from $-L+1$ to $L-1$. 
We use a Taylor expansion of $V((i +j)\delta_n)$  at $ (i\delta_n)$ and of order $q \leq 2D$: 
\begin{align}\label{e:zaza2}
V((i+j)\delta_n)=&V(i\delta_n)+\dots + 
\frac{(j\delta_n)^{q-1}}{(q-1)!}V^{(q-1)}(i \delta_n)
+(j\delta_n)^{q}\int_0^1 \frac{(1-\eta)^{q-1}}{(q-1)!}V^{(q)}((i +j\eta)\delta_n)d\eta.
\end{align}
Note that this  expression  is "telescopic"  in the sense 
 that  if $q < q' \leq 2D$,
 \begin{align}  \label{e:zaza3}
& (j\delta_n)^{q}\int_0^1 \frac{(1-\eta)^{q-1}}{(q-1)!}V^{(q)}((i +j\eta)\delta_n)d\eta \notag
\\
& =
 \frac{(j\delta_n)^{q}}{(q)!}V^{(q)}(i \delta_n)+ \dots + \frac{(j\delta_n)^{q'-1}}{(q'-1)!}V^{(q'-1)}(i \delta_n)
 +  (j\delta_n)^{q'}\int_0^1 \frac{(1-\eta)^{q'-1}}{(q'-1)!}V^{(q')}((i +j\eta)\delta_n)d\eta.
\end{align}

Combining \eqref{e:zaza2} (with $i=0$ and $q=2D$), the vanishing moments of the sequence $a^{2*}$ and $\left(\mathcal{H}_{1}\right)$ yields:
\begin{align*}
\E[V_{a,n}] 
=& n'\delta_n^{2D} R(0,\delta_n,2D,V^{(2D)},a^{2*})\\
=& n'C (-1)^D \delta_n^{2D+s} R(0,1,2D,\abs{\cdot{}}^s,a^{2*}) +n' \delta_n^{2D} R(0,\delta_n,2D,r,a^{2*}).
\end{align*}
The first term is non-zero by   \eqref{eq:sumneq0} in Proposition \ref{prop:sumneq0} and a dominated  convergence argument together with $\left(\mathcal{H}_{1}\right)$ shows that the last term is $o(\delta_n^{2D+s})$ giving \eqref{eq:esp_van_Dqqe}.

\medskip

{\bf 2)} Using Lemma \ref{lem:mehler}, \eqref{e:zaza2} with $q=2D$, the fact that $D\leq M$, and the vanishing moments of the sequence $a^{2*}$, we obtain
 \begin{align*}
\Var (V_{a,n}) 
 = & 2\sum_{i,i' =1} ^{n'}\Cov^2 \left(\Delta_{a,i}(X), \Delta_{a,i'}(X) \right)  
 =  2\sum_{i,i' =1} ^{n'} \left(-\Delta_{a^{2*},i-i'}(V)\right)^2
 =  2\sum_{i=-n'+1}^{n'-1} (n'-\abs{i})  \Delta_{a^{2*},i}(V)^2\\
 = & 2\delta_n^{4D}\sum_{i=-n'+1}^{n'-1} (n'-\abs{i})   R^2(i,\delta_n,2D,V^{(2D)},a^{2*})\\
 = & 2 \delta_n^{4D}\sum_{i=-n'+1}^{n'-1} (n'-\abs{i})  \left( C(-1)^D \delta_n^s R(i,1,2D,\abs{\cdot{}}^{s},a^{2*}) + R(i,\delta_n,2D,r,a^{2*})\right)^2.\\
 =:& A_n +B_n +C_n,
\end{align*}
where $B_n$ comes from  the double product.

\medskip

(i) We show that $A_n$ converges. Indeed,    
\begin{align*} 
A_n&= 2C^2 \delta_n^{4D}\sum_{i=-n'+1}^{n'-1} (n'-\abs{i})  \delta_n^{2s} R^2(i,1,2D,\abs{\cdot{}}^{s},a^{2*}) = 2C^2 n' \delta_n^{4D+2s} \sum_{i\in \Z} f_n(i),
 \end{align*}
 with
\[
f_n(i)\defeq \frac{n'-\abs{i}}{n'}  R^2(i,1,2D,\abs{\cdot{}}^{s},a^{2*})\ind_{\abs{i}\leq n'-1}.
\]
Since $f_n(i) \uparrow  R^2(i,1,2D,\abs{\cdot{}}^{s},a^{2*})$ for fixed $i$ and $n'$ going to infinity, it suffices to study  the series
\[
\sum_{i\in \Z} R^2(i,1,2D,\abs{\cdot{}}^{s},a^{2*}).
\]
Using \eqref{e:zaza3} , with $q'  = 2M$ ,  $\abs{\cdot{}}^{s}$ instead of $V^{(2D)}$ and $\delta_n =1$, and using the vanishing moments of the sequence $a^{2*}$, we get, for $i$ large enough so that  $i$ and $i+j$ always have the same sign in the sum below, 
\[
R(i,1,2D,\abs{\cdot{}}^{s},a^{2*})=
R(i,1,2M,g,a^{2*}) = 
 - \sum_{j} a^{2*}_j j^{2M} \int_0^1 \frac{(1-\eta)^{2M-1}}{(2M-1)!} g((i+j\eta))d\eta,
\]
where $g$ is the $2(M-D)$-th derivative of $\abs{\cdot{}}^{s}$ (defined on $\mathbb{R} \setminus \{0\}$).
For $i$ sufficiently large, $g(i +j \eta)$ is  bounded by $(Const) |i| ^{s-2(M-D)}$ so that 
\begin{equation}\label{e:a:0}
R^2(i,1,2D,\abs{\cdot{}}^{s},a^{2*}) \mbox{ is bounded by } (Const) i ^{2(s-2(M-D))},
\end{equation}
which is the general term of a convergent series.

%

\noi\\
(ii) Now we show that  the term $C_n$  is negligible compared to $A_n$. This will imply in turn that $B_n$ is negligible compared to $A_n$, from the Cauchy-Schwarz inequality. We have to give bounds  to the series with general term $R^2(i,\delta_n,2D,r, a^{2*})$ with 
\begin{align*}
R(i,\delta_n,2D,r, a^{2*})&=-
\sum_{j}  a_j^{2*} j^{2D} 
\int_0^1 \frac{(1-\eta)^{2D-1}}{(2D-1)!} r\left((i+j\eta)\delta_n\right) d\eta.
\end{align*}
For fixed $i$, the assumptions \eqref{e:cov} on $r$ in $\left(\mathcal{H}_{1}\right)$ are sufficient to build a dominated convergence argument  to  prove that $R^2(i,\delta_n,2D,r, a^{2*}) = o( \delta_n^{2s})$ which leads to the required result. So we concentrate our attention on indices $i$ such that $|i| > 2L$. Now we use $(\mathcal{H}_2)$ and the notation $d=2$ if $s<3/2$ and $d=3$ if $s \geq 3/2$. Consider $\beta$ as in $(\mathcal{H}_2)$ and remark that, in the mixed situation, we have $  s-d < \beta < -1/2 $. In the infill situation, it is assumed that for $|h| <1$, $|r^{(d)}|(h) \leq (Const) |h|^{\beta}$ with $\beta < -1/2$. Since $h$ is restricted to $[-1,1]$, we may also consider without loss of generality that $\beta$ has been chosen such that $  s-d < \beta < -1/2 $. Using \eqref {e:zaza3}  as  in the proof of item \tbf{1)}, if $2D+d\leq 2M$, one gets
\begin{align*}
R(i,\delta_n,2D,r, a^{2*})&=-
\sum_{j}  a_j^{2*} j^{2D+d}  \delta_n^d
\int_0^1 \frac{(1-\eta)^{2D+d-1}}{(2D+d-1)!} r^{(d)}\left((i+j\eta)\delta_n\right) d\eta.
\end{align*}
 The condition $|i| > 2L$ ensures that the integral is always convergent. Then we have
\begin{equation}\label{e:a:1}
 R^2(i,\delta_n,2D,r, a^{2*}) \leq (Const)   \delta_n^{2d+2\beta}  i^{2\beta}.
\end{equation}
 Since $\beta <-1/2$ , the series in $i$ converges and the contribution to $C$ of the indices $i$ such that 
  $|i| > 2L$  is bounded by
 $
 (Const)  \delta_n ^{4D +2d +2\beta}$ which is negligible compared  to $ \delta_n ^{4D +2s}$ since 
 $d +\beta >s$.  
\end{proof}

\begin{proof}[Proof of Theorem \ref{th:CLT_Van}] By a diagonalization argument,  $V_{a,n}$ can be written as

\[
V_{a,n} = \sum_{i=1}^{n''} \lambda_i Z_i^2,
\]
where $\lambda_1,\dots,\lambda_{n''}$ are the non-zero eigenvalues of variance-covariance matrix $\Sigma_a$ of $\mathbf{\Delta_a}(X)$ and the  $Z_i$ are  independent and identically distributed  standard Gaussian variables. 
Hence,
\begin{equation} \label{eq:CLT:Van}
\frac{ V_{a,n} - \mathbb{E}(V_{a,n}) }{ \sqrt{ \Var(V_{a,n}) } }
=
\sum_{i=1}^{n''}
 \frac{
 \lambda_i
 }{
 \sqrt{ \sum_{r=1}^{n''} \lambda_r^2 } 
   } (Z_i^2 - 1).
\end{equation}
In such a situation, Lemma 2 in \cite{IL97} implies that the Lindeberg condition
is a  sufficient  condition required to prove the central limit theorem and is equivalent to 
\begin{align*}
\max_{i=1,\dots,n''} | \lambda_i |
=
o\left( \sqrt{ \Var( V_{a,n} ) } \right).
\end{align*} 
From Lemma \ref{lem:cov}, one has 
\[
\max_{i=1,\dots,n''}
\left(
\sum_{j=1}^{n''}
| \Sigma_{a}(i,j) |
\right)
=
o\left(  \sqrt{ \sum_{r=1}^{n''} \lambda_r^2 }  \right)
\]
and the result follows using the following classical linear algebra result (see for instance \cite[Ch. 6.2, p194]{Luenberger79})
\begin{align*}
\max_{i=1,\dots,n''} | \lambda_i |
& \leq
\max_{i=1,\dots,n'}
\left(
\sum_{j=1}^{n'}
| \Sigma_{a}(i,j) |
\right).
\end{align*}
\end{proof}

\begin{proof}[Proof of Corollary \ref{cor:CLT_Van_joint}]
To prove the asymptotic  joint normality it is sufficient to prove the asymptotic normality  of any non-zero linear  combination
\[
LC(\gamma)=
\sum_{j=1}^k
\gamma_j
V_{a^{(j)},n},
\]
where $\gamma_j\in \R$ for $j=1,\dots,k$. We have again the representation 
\[
LC(\gamma) =  \sum_{i=1}^{n''} \lambda_i Z_i^2,
\] where the $\lambda_i$'s are now the non-zero eigenvalues of the variance-covariance matrix
\[
\sigma' =
\sum_{j=1}^k
\gamma_j   \Sigma_{a^{(j)},n},
\] and the $Z_i$'s are as before. The Lindeberg condition has the same expression.
On one hand, as $n$ goes to infinity,
\[
\frac{1}
{n\delta_n^{4D+2s}}  \sum_{i=1}^{n''} \lambda_i  \to \gamma^\top \Lambda_\infty  \gamma
\]
where $^{\top}$ stands for the transpose.  On the other  hand,  by the triangular inequality for the operator norm (which is the maximum of the $|\lambda_i|$'s), one gets
 \[
 \max_{i=1,\dots,n''} | \lambda_i |  =   \|\sigma'\|_{op}
 \leq \sum_{j=1}^k
\gamma_j   \|\Sigma_{a^{(j)},n}
\|_{op}.
\]
In the proof of Theorem \ref{th:CLT_Van}, we have established that $ \|\Sigma_{a^{(j)},n}
\|_{op} = o( n^{1/2}\delta_n^{2D+s})$ leading to the result. 
\end{proof}

\subsection{Proof of the remaining results in Section \ref{sec:quad_a_var}}

\begin{proof}[Proof of Theorem \ref{th:CLT_Can}]
We use the definition of $C_{a,n}$ and the following decomposition:
\begin{align*}
\frac{C_{a,n} -C}{\sqrt{\Var(C_{a,n})}} =\frac{C_{a,n} -\E[C_{a,n}]}{\sqrt{\Var(C_{a,n})}} + \frac{\E[C_{a,n}]-C}{\sqrt{\Var(C_{a,n})}} =
\frac{V_{a,n} -\E[V_{a,n}]}{\sqrt{\Var(V_{a,n})}} + \frac{\E[C_{a,n}]-C}{\sqrt{\Var(C_{a,n})}}.
\end{align*}
Following the proof of Proposition \ref{prop:Van_Dqqe}, the second term is proportional to 
\[
\sqrt n \delta_n^{-s} R(0,\delta_n,2D,r,a^{2*})=-\sqrt n \delta_n^{-s} \sum_{i} a^{2*}_i i^{2D}\int_0^1 \frac{(1-\eta)^{2D-1}}{(2D-1)!}r(i\eta \delta_n) d\eta
\]
which is negligible compared to $(Const) n^{1/2+\alpha s - \alpha (s+1/2 \alpha)} $ by 
$\left(\mathcal{H}_{3}\right)$ and thus goes
to 0 as $n$ goes to infinity. Then Slutsky's lemma and Theorem \ref{th:CLT_Van} lead straightforwardly to the required result.
\end{proof}

\begin{proof}[Proof of Corollary \ref{cor:cr}]
Obviously, one has   
\[
V_{a,n}^X = \| \mathbf{\Delta_{a}} (X) \|^2 = 
\|\mathbf{\Delta_{a}} ( f)  +\mathbf{\Delta_{a}} (\overline X) \|^2.
\]
Using the triangular  inequality $ \|A+B\|^2-\| A\|^2 \leq \| B\|^2 + 2\| A\| \| B\|$, it suffices  to have 
$\|\mathbf{\Delta_{a}} (f) \|^2 =  o(\Var (V_{a,n} (\overline X)^{1/2}) = o( n^{1/2} \delta_n^{2D +s})$ to deduce  the central limit theorem for $X$ from that for $\overline X $. By application of the Taylor-Lagrange formula, one gets
\[
\Delta_{a,i}(f) = (Const)\times \delta_n^M \times  f^{(M)} (\xi),
\]
with $\xi \in [0, n^{1-\alpha}]$.  Then  $\|\mathbf{\Delta_{a}} (f) \| ^2 \leq n( K^\alpha_{M,n})^2  \delta_n^{2M}$ and 
 a sufficient condition is \eqref{eq:cond_K}.
\end{proof}

\paragraph{Acknowledgments}
This work has been partially supported by the French National
Research Agency (ANR) through project PEPITO
(no ANR-14-CE23-0011). The authors are grateful to Laurent Risser, for providing them the image registration data of Section \ref{subsubsection:large:data:set}.

\bibliographystyle{abbrv}
\bibliography{biblio_variation}
\end{document}